\documentclass[10pt]{article}
\usepackage{latexsym,amsfonts,amsmath,amssymb,array,tabularx,theorem}
\usepackage{graphicx}
\usepackage{caption}
\usepackage{subcaption}
\usepackage[normalem]{ulem}
\usepackage[pdfpagelabels]{hyperref}
\usepackage{bbm}
\usepackage{authblk}
\usepackage{bm}
\usepackage{mathtools} %\bigtimes
\usepackage{dsfont} %\mathbbm{1}
\usepackage{algpseudocode}
\usepackage[section]{algorithm}

\usepackage[normalem]{ulem}

\usepackage{graphicx}
\usepackage{epstopdf}
\usepackage[all]{xy}

\newcommand{\wbeta}{\underline{\beta}}
\newcommand{\wb}{\underline{b}}

\newcommand{\dcup}{\stackrel{.}{\cup}}
\usepackage{float}
\newtheorem{theorem}{Theorem}[section]

\newtheorem{proposition}[theorem]{Proposition}
\newtheorem{remark}[theorem]{Remark}
\newtheorem{definition}[theorem]{Definition}

\newenvironment{proof}{\begin{trivlist}
    \item[\hskip\labelsep{\bf Proof.}]}{$\hfill\Box$\end{trivlist}}

{\theoremstyle{plain} \theorembodyfont{\rmfamily}
}

\numberwithin{equation}{section}
\numberwithin{figure}{section}

\allowdisplaybreaks[1]

\newcommand{\bsvarrho}{{\boldsymbol{\varrho}}}

\newcommand{\bbR}{\mathbb{R}}

\newcommand{\R}{\mathbb{R}}

\newcommand{\N}{\mathbb{N}}

\newcommand{\bbP}{\mathbb{P}}

\newcommand{\calG}{\mathcal{G}}

\newcommand{\calN}{\mathcal{N}}
\newcommand{\calNN}{\mathcal{NN}}

\newcommand{\calT}{\mathcal{T}}

\newcommand{\calV}{\mathcal{V}}

\newcommand{\NN}[1]{\left\|#1\right\|}
\newcommand{\norm}[2][]{\| #2 \|_{#1}}

\newcommand{\snorm}[2][]{| #2 |_{#1}}

\newcommand{\ceil}[1]{\lceil #1 \rceil}

\newcommand{\setc}[2]{\left\{#1\, :\,#2\right\}}

\newcommand{\domain}{{\mathrm{D}}}

\newcommand{\Sop}[1]{\operatorname{S}_{#1}}

\newcommand{\So}{\Sop{1}}

\newcommand{\be}{\begin{equation}}
\newcommand{\ee}{\end{equation}}
\newcommand{\bea}{\begin{eqnarray}}
\newcommand{\eea}{\end{eqnarray}}
\newcommand{\beas}{\begin{eqnarray*}}
\newcommand{\eeas}{\end{eqnarray*}}

\newcommand{\relu}{{\rho}}

\newcommand{\realiz}[1]{{\rm R}(#1)} 
 
\newcommand{\depth}{L}
\newcommand{\size}{M}
\newcommand{\Parallel}[1]{{\rm P}(#1)} 
\newcommand{\FParallel}[1]{{\rm FP}(#1)}
\newcommand{\Parallelc}[1]{{\rm P}\left(#1\right)} 
\newcommand{\FParallelc}[1]{{\rm FP}\left(#1\right)}
\newcommand{\sconc}{\odot}
 
\newcommand{\sizefirst}{\size_{\operatorname{in}}}
\newcommand{\sizelast}{\size_{\operatorname{out}}}

\newcommand{\inradius}{r}
\DeclareMathOperator{\conv}{conv}
\DeclareMathOperator{\interior}{int}

\newcommand{\Cheb}{{Cheby\v{s}ev} }

\newcommand{\Id}{{\rm Id}}

\newcommand{\ReLUtwo}{ReLU$^2$ }
\newcommand{\ReLUtwonospace}{ReLU$^2$}
\newcommand{\ReLU}{ReLU }
\newcommand{\ReLUnospace}{ReLU}
\newcommand{\BiSU}{BiSU }
\newcommand{\BiSUnospace}{BiSU}

\renewcommand{\S}{\mathcal S}
\newcommand{\M}{\mathcal M}
\newcommand{\LLL}{\mathfrak L}
\newcommand{\MMM}{\mathfrak M}

\newcommand{\OOO}{\mathfrak O}
\newcommand{\TTT}{\mathfrak T}
\newcommand{\dist}{\mathrm{dist}}

\newcommand{\OO}{\domain_0}

\newcommand{\oc}{\omega_c}

\begin{document}
\bibliographystyle{abbrv}
\title{
Exponential Expressivity of \ReLUnospace$^k$ Neural Networks
\\
on Gevrey Classes with Point Singularities
}
\author[$\dagger$]{Joost A. A. Opschoor} 
\author[$\dagger$]{Christoph Schwab}
\affil[$\dagger$]{\footnotesize Seminar for Applied Mathematics, ETH
  Z\"{u}rich, R\"{a}mistrasse 101, CH--8092 Z\"urich, Switzerland.
    joost.opschoor@sam.math.ethz.ch,\;
    christoph.schwab@sam.math.ethz.ch}
\maketitle
\date{}
\begin{abstract}

We analyze deep Neural Network emulation rates 
of smooth functions with point singularities in
bounded, polytopal domains $\domain \subset \R^d$,
$d=2,3$. 
We prove exponential emulation rates 
in Sobolev spaces in terms of the 
number of neurons and in terms of the number of nonzero coefficients
for Gevrey-regular solution classes defined in terms 
of weighted Sobolev scales in $\domain$, 
comprising the countably-normed spaces of I.M. Babu\v{s}ka and B.Q. Guo.

As intermediate result,
we prove that continuous, piecewise polynomial
high order (``$p$-version'') finite elements
with elementwise polynomial degree $p\in\N$
on arbitrary, regular, simplicial partitions of polyhedral domains 
$\domain \subset \R^d$, $d\geq 2$
can be \emph{exactly emulated} by 
neural networks combining \ReLU and \ReLUtwo activations.
On shape-regular, simplicial partitions of polytopal domains $\domain$, 
both the number of neurons and the number of nonzero parameters 
are proportional to the number of degrees of freedom of the finite element space,
in particular for the $hp$-Finite Element Method of I.M. Babu\v{s}ka and B.Q. Guo.
\end{abstract}



%%%%%%%%%%%%%%%%%%%%%%%%%%%%%%%%%%%%%%%%%%%%%%%%%%%%%%%%%%%%%%%%%
\section{Introduction}
\label{sec:Intro}
%%%%%%%%%%%%%%%%%%%%%%%%%%%%%%%%%%%%%%%%%%%%%%%%%%%%%%%%%%%%%%%%%%

Recent years have seen the emergence of machine learning 
in scientific computing. 
One key component are \emph{Deep Neural Network} 
(DNN for short) based numerical approximations of functions
and operators. 
DNNs appear to impinge on
nearly all applications of computation in science 
and engineering, including the numerical solution of 
Partial Differential Equations (PDEs).
It is therefore of interest to explore, mathematically, the 
consequences of this development for established methodologies
in numerical analysis and scientific computing. 
The present paper addresses DNNs in the context of
the $hp$-version of the Finite Element Method,
as initiated by I.M. Babu\v{s}ka and B.A. Szab\'o \cite{BabSzBook} and
coworkers.
It will, in particular, 
develop \emph{exact DNN emulations of $hp$-finite element (FE for short)
spaces in polygonal and polyhedral domains $\domain$}. 
As a natural consequence, 
\emph{existing exponential convergence rate results} 
for $hp$-FE approximation of elliptic boundary value problems
(e.g. \cite{ApproxhpFE,BabGuoCurved,Schwabphp,MelenkBook,BabSzBook,MelSch98} and the references there)
will imply corresponding DNN approximation error bounds.
%%%%%%%%%%%%%%%%%%%%%%%%%%%%%%%%%%%%%%%%%%%%%%%%%%%%%%%%%%%%%%%%%
\subsection{Previous Work}
\label{sec:PreWrk}
%%%%%%%%%%%%%%%%%%%%%%%%%%%%%%%%%%%%%%%%%%%%%%%%%%%%%%%%%%%%%%%%%

In a precursor to this manuscript, \cite{OPS2020}, 
we constructed, in particular,
deep ReLU neural network (NN for short)
emulations of univariate $hp$-FE spaces (referred to also as
        ``variable degree, variable knot splines''). 
In \cite{OS2023}, an improved NN architecture to express high-order 
polynomials by \ReLU NNs based on the emulation of \Cheb polynomials was
developed, implying improved stability and 
expression rate bounds in some of the results in \cite{OPS2020}.

In \cite{HLXZ2020}, \ReLU DNN emulations of first order, 
Lagrangean FE spaces (``Courant Finite Elements'') 
on particular, regular triangulations of 
polytopal domains where constructed, and first bounds on size and depth 
of the NNs were proved. 
Admissible simplicial partitions in  \cite{HLXZ2020} 
were subject to certain patch conditions which were essential 
in the DNN architecture. 

The first \ReLU NN emulation of first order, Lagrangean FE spaces
on general regular simplicial partitions in any space dimension 
were constructed in \cite[Sec.~4]{LODSZ22_991}.
The present results build, in part, on the constructions in 
\cite{LODSZ22_991}.

Accordingly, 
the recent \cite{HX2023} is closely related to the present work.
There, using the construction in \cite[Sec.~4]{LODSZ22_991},
DNN emulations of arbitrary order, Lagrangean FE spaces on 
general, regular simplicial partitions are developed,
via DNNs which are \ReLU and \ReLUtwo activated. 
The constructions in \cite{HX2023} furnish exact DNN representations
of Lagrangean FE spaces on regular, simplicial triangulations of $\domain$, 
by emulating in the feature space of the DNNs the nodal, Lagrangian
basis of arbitrary, fixed polynomial degree $p\geq 1$.
%%%%%%%%%%%%%%%%%%%%%%%%%%%%%%%%%%%%%%%%%%%%%%%%%%%%%%%%%%%%%%%%%
\subsubsection{Finite Elements }
\label{sec:PreFinElts}
%%%%%%%%%%%%%%%%%%%%%%%%%%%%%%%%%%%%%%%%%%%%%%%%%%%%%%%%%%%%%%%%%

Founded in the mid-20th century,
Finite Element 
discretizations are nowadays
well-established discretization methodologies, 
in particular for elliptic and parabolic PDEs.
Their mathematical underpinning 
(see, e.g., the text \cite{EGI} and the references there) has, 
to some extent, an accepted
set of terminology and notation, which we recapitulate here.

On a polytopal domain $\domain$ of dimension $d\geq 2$,
we consider regular, simplicial partitions $\mathcal{T}$.
We start from the \ReLU neural network 
emulation of continuous, piecewise linear functions,
forming the classical ``Courant Finite Elements'',
whose space is denoted by $\So(\mathcal{T},\domain)$,
and consider the exact emulation of high-order, Lagrangean
Finite Element
spaces, 
for polynomial degree $p\in\N = \{ 1,2,3,...\}$ 
defined as
\begin{align*}
\Sop{p}(\calT,\domain) 
	= &\, \{ v \in C^0(\domain) : v|_K \in \bbP_p(K) \text{ for all } K\in\calT \}
	.
\end{align*}
We recall from \cite[Sec. 1.4.1]{LODSZ22_991}
notation for the used partitions.
For
$k\in\{0,\ldots,d\}$ we define a $ k $-simplex $ K $ by
$ K = \conv(\{a_0,\ldots, a_k\}) \subset \R^d $, for some
$ a_0,\ldots,a_k \in \R^d $ which do not all lie in one affine
subspace of dimension $k-1$, and where
\begin{equation*}
\conv (Y) 
:= 
\setc{x = \sum_{y \in Y}\lambda_y y}{ \lambda_y\geq0 \, \text{ and }\, \sum_{y \in Y} \lambda_y = 1 }
\end{equation*}
denotes the closed convex hull.\footnote{
In \cite{LODSZ22_991}, 
$\conv$ denotes the \emph{open convex hull}.
As a result, in \cite{LODSZ22_991} simplices are open by definition.
Here, simplices are closed by definition.}
By $\snorm{K}$ we denote the $k$-dimensional Lebesgue measure of 
the $k$-simplex $K$.
We consider a 
simplicial partition $ \calT $ of a polytopal, 
bounded domain $ \domain \subset \R^d$ into $d$-simplices,
i.e. 
$ \overline{\domain} = \bigcup_{K\in \calT} K $ and
$ \interior K\cap \interior K' = \emptyset$, for all $ K\neq K' $.
We assume that $ \calT $ is a \emph{regular} partition, 
i.e. for all distinct $K, K' \in\calT$ 
it holds that $K\cap K'$ 
is a $k$-subsimplex of $K$ 
for some $k\in\{0,\ldots,d-1\}$.
I.e., 
there exist
$a_0,\ldots,a_d \in \overline\domain$ 
such that 
$K = \conv(\{a_0,\ldots,a_d\})$
and
$K\cap K' = \conv(\{a_0,\ldots,a_k\})$.
 
Recall that the 
\emph{shape-regularity constant} 
$\kappa_{\rm sh}(\calT)$
of a simplicial partition $\calT$ of $\domain$ 
is
$ \kappa_{\rm sh} 
:=
\max_{K\in \calT}\tfrac{h_K }{\inradius_K} > 1 $.  
Here
$h_K := \operatorname{diam}(K)$ 
and 
$\inradius_K$ is the radius of the
largest ball contained in $K$.
Let $ \calV $ be the set of vertices of $ \calT $.  
For all $i \in \calV$ we denote by
$s(i) := \snorm{ \{ K \in \calT : i \in K \} }$
the number of elements of $\calT$ sharing the vertex $i$,
and define
$\mathfrak{s}(\calV) := \max_{i\in\calV} s(i)$.
Throughout, we will use the notation $\snorm{S}$ 
for the cardinality of a finite set $S$.

\begin{remark}
\label{rem:mathfraks}
The constant $\mathfrak{s}(\calV)$ 
can be bounded in terms of the 
shape regularity constant $\kappa_{\rm sh}(\calT)$
(see \cite[Rem. 11.5 and Prop. 11.6]{EGI}, 
which generalize verbatim to dimension $d>3$).
\end{remark}
%%%%%%%%%%%%%%%%%%%%%%%%%%%%%%%%%%%%%%%%%%%%%%%%%%%%%%%%%%%%%%%%%
\subsubsection{Exact NN Emulation of Lagrangean, Nodal Finite Elements}
\label{sec:PreExactEmul}
%%%%%%%%%%%%%%%%%%%%%%%%%%%%%%%%%%%%%%%%%%%%%%%%%%%%%%%%%%%%%%%%%
For a polytopal domain $\domain$ and a regular, simplicial partition $\calT$,
an exact emulation of $\So(\mathcal{T},\domain)$
by ReLU NNs is provided in \cite{ABMM2016},
but without efficient bounds on the NN depth and size.
Such bounds were first obtained
in \cite[Sec. 3]{HLXZ2020},
under the assumption that the mesh has ``convex patches'',
i.e., 
it was assumed that for all $i \in \calV$ 
the set $\{ K \in \calT : i \in K \}$ is convex.
It was shown that 
functions in $\So(\mathcal{T},\domain)$ can be emulated exactly by a ReLU NN
of depth independent of $N = \dim( \So(\mathcal{T},\domain) )$
and size growing at most linearly in $N$,
with constants depending on $d$ and the shape regularity of the mesh.
By a geometrical construction 
in which possibly nonconvex patches 
are re-written as the finite union of convex patches,
the results from \cite[Sec. 3]{HLXZ2020}
were extended to general regular, simplicial meshes,
without assuming convexity of patches,
in \cite{LODSZ22_991}.
There, 
it was shown that also without assuming convexity of patches
the depth is independent of $N = \dim( \So(\mathcal{T},\domain) )$
and the size grows at most linearly in $N$,
with constants depending on $d$ and the shape regularity.

For the exact emulation of high order finite elements
$\Sop{p}(\calT,\domain)$ for $p\in\N$,
a first result using \ReLUnospace, \ReLUtwo and \emph{binary step unit}
(\emph{\BiSUnospace}) activations
was given in \cite[Sec. 7.1]{LODSZ22_991}.
The \ReLUnospace, \ReLUtwo and \BiSU activation functions $\R\to\R$ 
are defined by 
$x\mapsto \max\{0,x\}$,
$x\mapsto \max\{0,x\}^2$
and
$x \mapsto 1$ for $x>0$ and $x\mapsto 0$ for $x\leq 0$,
respectively.
The networks constructed in \cite[Sec. 7.1]{LODSZ22_991}
have depth bounded by $C d\log_2(p+1)$
and size bounded by $C \snorm\calT (p+1)^d$,
which is proportional to the number of degrees of freedom
(i.e. the dimension) of the FE space.
Although this emulation result is efficient, 
it is not completely satisfactory
as the use of discontinuous \BiSU activations
for the emulation of continuous FE basis
functions could be considered somewhat unnatural.
The use of \ReLUtwo activations 
for the exact emulation of high order finite elements is natural, 
as it allows the exact emulation of products (cf. Prop. \ref{prop:nnprod}).

The exact emulation of $\Sop{p}(\calT,\domain)$
by NNs with only \ReLU and \ReLUtwo activations 
was first given in \cite{HX2023}.
The main insight is that on an element $K \in \calT$,
the ``hat'' basis functions of $\So(\mathcal{T},\domain)$,
which equal $1$ in one vertex $i\in\calV$ and vanish in all other vertices in $\calV$,
coincide with the barycentric coordinates on $K$.
Elements of $\Sop{p}(\calT,\domain)$ 
can therefore be written as sums of products of ``hat'' basis functions.
In \cite{HX2023}, the NNs from \cite{LODSZ22_991}
which exactly emulate $\So(\mathcal{T},\domain)$
are combined with product subnetworks
to obtain an exact emulation of elements in $\Sop{p}(\calT,\domain)$.
The size of the networks constructed in \cite{HX2023}
is larger than the number of degrees of freedom 
$\dim(\Sop{p}(\calT,\domain))$,
which is suboptimal, as we will show (see Rem.~\ref{rmk:pbds}).
A detailed discussion of the results in \cite{HX2023}
and the complexity of the NNs constructed there
is the topic of Sect. \ref{sec:previoushofem}.
%%%%%%%%%%%%%%%%%%%%%%%%%%%%%%%%%%%%%%%%%%%%%%%%%%%%%%%%%%%%%%%%%
\subsubsection{NN Emulation of $hp$-Finite Elements}
\label{sec:PrehpNN}
%%%%%%%%%%%%%%%%%%%%%%%%%%%%%%%%%%%%%%%%%%%%%%%%%%%%%%%%%%%%%%%%%

The literature on 
exponential convergence of NN approximations of functions with point singularities
based on the NN approximation of $hp$-Finite Elements
goes back to \cite{OPS2020}, 
where it is shown that there exist 
exponentially convergent ReLU NN approximations 
of univariate weighted Gevrey regular functions
(see Sect. \ref{sec:wgtSobSpc} for a definition of weighted Gevrey regular functions).
These functions, which are defined on a bounded interval,
are smooth everywhere except in a finite set of singular points.
In that work, only \ReLU activations were used,
i.e. multiplications could not be realized exactly 
(as \ReLU NNs realize continuous, piecewise linear functions)
and were approximated using networks introduced in \cite{yarotsky}
whose size depends logarithmically on their accuracy.
Approximating in this way 
all products which are needed for the approximation of high order finite elements
results in a NN whose size grows more quickly with the polynomial degree
than the number of degrees of freedom 
of the $hp$-Finite Elements which it approximates.

For $d=2,3$, 
for weighted analytic functions on polygons with point singularities 
and on polyhedra with point- and edge singularities,
the existence of exponentially convergent \ReLU NN approximations 
was shown in \cite{MOPS23}.
Again, multiplications were approximated by \ReLU subnetworks,
leading to an inexact approximation of $hp$-Finite Elements
and a network size that is larger than the number of degrees of freedom.
For the approximation of weighted analytic functions on $(0,1)^d$,
the networks constructed in \cite{MOPS23}
are based on tensor product $hp$-Finite Elements 
on tensor product meshes with rectangular elements.
To approximate weighted analytic functions on domains
which do not have a tensor product structure,
multiple such approximations are combined using a partition of unity.

While in \cite{MOPS23}, exponentially convergent networks are shown to exist,
it is not clear how network parameters (weights and biases)
that realize this exponential convergence
can be computed based on a finite number of function evaluations.
This is the topic of \cite{JOdiss}.
Exponentially convergent \ReLU NNs are constructed
which approximate weighted analytic functions on polygons
that have point singularities in the vertices of the domain.
The NN construction in \cite{JOdiss} is similar to, 
but different from that in \cite{MOPS23}.
The main difference is that local polynomial approximations
on rectangular elements in tensor product partitions of $(0,1)^d$,
which make up the $hp$-Finite Elements on $(0,1)^d$,
are computed by nodal interpolation,
rather than the tensor product $H^1$ projection used in \cite{MOPS23}.
This means that the network parameters can be computed explicitly
based on a finite number of queries of the approximated function,
whose number grows at most polylogarithmically with the accuracy.
Also, it means that we can use 
existing bounds on the Lebesgue constants for polynomial interpolation
to show stability of the NN approximations.

%%%%%%%%%%%%%%%%%%%%%%%%%%%%%%%%%%%%%%%%%%%%%%%%%%%%%%%%%%%%%%%%%
\subsection{Contributions}
\label{sec:Contr}
%%%%%%%%%%%%%%%%%%%%%%%%%%%%%%%%%%%%%%%%%%%%%%%%%%%%%%%%%%%%%%%%%
%
In our main result, Thm. \ref{thm:nngevrey},
we establish, for functions from weighted analytic 
and weighted Gevrey regular spaces
with point singularities in bounded, polytopal domains $\domain$ 
in Euclidean space of dimension $d=2$ or $d=3$, 
\emph{exponential expressivity bounds of certain deep neural networks}.
The NNs under consideration are deep feedforward NNs 
which encode so-called $hp$ Finite Element approximations
on shape-regular, simplicial partitions of $\domain$ with geometric
refinement(s) towards the singular support of the function to be
emulated.
We consider singular support sets $\S\subset \partial\domain$
which are contained in the boundary.\footnote{The present analysis
also covers point singularities in the interior of $\domain$,
when they are located at a vertex of the used 
regular, simplicial triangulation of the domain.}

Compared to \cite{MOPS23,JOdiss},
our construction combines both \ReLU and \ReLUtwo activations,
rather than only \ReLU activations.
This means that products of real numbers can be emulated exactly
and allows us to obtain a NN size which is 
proportional to the number of degrees of freedom
in the Finite-Element space.
The use of two different activation functions throughout our networks
means that the used architectures are slightly more complex than pure \ReLU networks,
as for every node in the network, 
the choice of activation function has to be specified.
As in \cite{YWYX2023}, 
in each layer, the same activation is applied in all positions in that layer.
The network consists of several layers with \ReLU activation, 
followed by several layers with \ReLUtwo activation. 
The ReLU layers emulate continuous, piecewise linear hat functions on a regular, simplicial triangulation. 
The number of such layers depends on the shape regularity of the triangulation. 
It depends logarithmically on the dimension $d$ of the domain, but it is independent of the polynomial degree $p$.
The number or \ReLUtwo layers is independent of the triangulation and its shape regularity, 
but depends logarithmically on $p$ and on $d$.

Our construction works with arbitrary regular, simplicial partitions,
which makes it simpler than 
the partition of unity-$hp$-constructions in \cite{MOPS23,JOdiss}.
We do not provide exponential rate bounds to approximate
edge singularities in three space dimensions,
but extend previous convergence rate bounds
which only hold for weighted analytic functions 
to weighted Gevrey classes with isolated point singularities.
Isolated point
singularities arise in space dimension $d=2$ as corner singularities in
solutions to elliptic BVPs with analytic data in polygons \cite{BabGuoRegI,GB1993,GS2006},
and in space dimension $d=3$ in a number of applications
\cite{CCM2010,MadMar19,FHHS2009,SGKM1995}.

Our networks are constructed by emulating 
continuous, piecewise polynomial nodal Lagrangean basis functions 
(see Rem. \ref{rem:nodalinterpolation}).
Therefore, 
as in \cite{HOS2022,JOdiss},
nodal interpolation could be used 
to construct NN approximations and explicitly compute their parameters (weights and biases),
provided that the function to be approximated can be queried.
We could compute the parameters 
based on only a finite number of function evaluations,
whose number equals the number of degrees of freedom
of the underlying $hp$ Finite Elements
and grows polylogarithmically with the accuracy.

The NNs constructed in Sect. \ref{sec:efficienthofem},
which exactly emulate high order finite elements,
are of independent interest, 
also beyond the current context of $hp$-Finite Elements.
We exactly emulate any continuous, piecewise polynomial function
of degree $p\in\N$
on a regular, simplicial partition of a polytopal domain $\domain \subset \R^d$ for $d\geq2$.
We show that on shape regular meshes,
the network size is proportional to the number of degrees of freedom,
see Prop. \ref{prop:reluhofem} and its discussion in Rem. \ref{rem:sizeproptodof}.
Hereby, we improve upon \cite{LODSZ22_991}
in the sense that we obtain exact emulations with the same order of computational complexity,
but without using discontinuous \BiSU activations.
As compared to \cite{HX2023}, 
the present NN architectures achieve this with smaller network sizes.
%%%%%%%%%%%%%%%%%%%%%%%%%%%%%%%%%%%%%%%%%%%%%%%%%%%%%%%%%%%%%%%%%
\subsection{Layout}
\label{sec:outline}
%%%%%%%%%%%%%%%%%%%%%%%%%%%%%%%%%%%%%%%%%%%%%%%%%%%%%%%%%%%%%%%%%
The structure of this paper is as follows.  
In Sect. \ref{sec:wgtSobSpc}, we recall weighted function classes 
for characterizing the smoothness of functions with point singularities
and introduce classes of weighted Gevrey regular functions,
which include as a special case ($\delta=1$) weighted analytic functions.
In Sect. \ref{sec:hpAppr}, we recall $hp$-approximations of such functions
and their exponential convergence.
In Sect.~\ref{sec:nnbasics}, we fix notation and recall basic
terminology and results from deep neural networks, as needed 
subsequently.
The topic of Sect. \ref{sec:hofem}
is the exact emulation of high order finite elements by neural networks.
In Sect. \ref{sec:previoushofem}, we discuss previous results from \cite{HX2023}.
Based on an observation from \cite{HX2023},
a new and more efficient NN emulation is constructed in Sect. \ref{sec:efficienthofem}.
This is used in Sect. \ref{sec:nnhpApprox}
to obtain exponentially convergent NN approximations 
of Gevrey regular functions in polygonal or polyhedral domains
with point singularities.
Sect. \ref{sec:Concl} concludes the paper
with an overview of the main results 
and a discussion of the conditioning of the used Lagrangean finite element basis.
%%%%%%%%%%%%%%%%%%%%%%%%%%%%%%%%%%%%%%%%%%%%%%%%%%%%%%%%%%%%%%%%%%%%%%%%%%%%%%%%%%
\section{Regularity and $hp$-Approximation of Point Singularities}
\label{sec:model}
%%%%%%%%%%%%%%%%%%%%%%%%%%%%%%%%%%%%%%%%%%%%%%%%%%%%%%%%%%%%%%%%%%%%%%%%%%%%%%%%%%
We review definitions and results from weighted, countably-normed function classes, 
as used in $hp$-FE approximation theory, and in the corresponding elliptic regularity.
These corner-weighted, analytic classes were introduced in the late 80ies 
by I.M. Babu\v{s}ka and B.Q. Guo and 
by P. Bolley, J. Camus and M. Dauge in their pioneering works
\cite{GuoDiss85,BDC85,BabGuoRegI,BabGuoRegII,GB1993,ApproxhpFE} 
and the references there in space dimension $d=2$,
and in space dimension $d=3$ in \cite{GuoBab3dI,GuoBab3dII,CoDaNi2012}.

Throughout, $\domain\subset \R^d$ denotes an 
open, bounded, polytopal domain in Euclidean space of dimension $d=2,3$.
We denote by~$\S \subset \partial\domain$ 
a finite set of singular points.
We consider solutions $u\in H^1(\domain)$ 
which are smooth in $\overline{\domain}\backslash \S$ 
so that the singular support of $u$ coincides with $\S$.
This allows us to 
determine in $\domain$ a collection of $|\S|$ many 
disjoint open sets $\omega_c\subset \domain$ 
containing exactly one singularity $ c \in \S$ in their 
closure.
We denote
$\domain_0 :=\domain \backslash \overline{\bigcup_{c \in \S} \omega_c}$.
%%%%%%%%%%%%%%%%%%%%%%%%%%%%%%%%%%%%%%%%%%%%%%%%%%%%%%%%%%%%%%%%%%%%%%%%%%
\subsection{Weighted Sobolev Spaces and Gevrey Classes}
\label{sec:wgtSobSpc}
%%%%%%%%%%%%%%%%%%%%%%%%%%%%%%%%%%%%%%%%%%%%%%%%%%%%%%%%%%%%%%%%%%%%%%%%%%
We characterize analytic regularity of singular solutions
by weighted Sobolev spaces. 
To define these, we follow \cite[Sect.~2.1]{hpfem} 
and introduce distance functions to a corner point $c\in \S$:
\be\label{eq:radii}
r_c(x) = \dist(x, c)\;,\qquad x\in \domain\;,\quad c\in \S\;.
\ee
For each conical point $c\in\S$, a \emph{singular exponent}
$\beta_c \in \mathbb{R}$
quantifies the allowed strength of the singularity at $c$.
We collect all singular exponents
in the ``multi-weight-exponent''
\begin{equation}\label{eq:weight}
\wbeta=\{\beta_c:\, c\in\S\} \in \mathbb{R}^{|\S|}\;.
\end{equation}
We assume for $d=3$ ($\wbeta > s $ and $\wbeta\pm s$ being
understood componentwise for $s\in \mathbb{R}$) 
\begin{equation}\label{eq:beta}
\wb := - 1 - \wbeta \in (0,1/2)\;, \;\;\mbox{ie.} \;\; -3/2 \; <  \wbeta < \; -1 
\;.
\end{equation}
For $d=2$, we assume for some $\varepsilon>0$ that
\begin{equation}\label{eq:beta2}
\wb := - 1 - \wbeta \in (0,\varepsilon)\;, \;\;\mbox{ie.} \;\; -1-\varepsilon \; <  \wbeta < \; -1  
\;.
\end{equation}
We consider the \emph{inhomogeneous, corner-weighted semi-norms}
$|u|_{N^k_{\wbeta}(\domain)}$ given by
(cf.~\cite[Def.~6.2 and Eq.~(6.9)]{CoDaNi2012}, 
\cite{GB1993} and \cite{GuoBab3dI}),
\be\label{eq:sn}
|u|^2_{N^k_{\wbeta}(\domain)} 
= 
|u|^2_{H^k(\OO)}
+
\sum_{c\in\S}
\sum_{\genfrac{}{}{0pt}{}{\alpha\in\mathbb{N}^d_0}{|\alpha|=k}}
\big\|{r_c^{\max\{\beta_c + |\alpha|,0\}}D^{\alpha}u}\big\|^2_{L^2(\oc)}
\;, \;\; k\in \mathbb{N}_0 \;.
\ee
We define the 
\emph{inhomogeneous weighted norm} $\|u \|_{N^m_{\wbeta}(\domain)}$
by
$\NN{u}^2_{N^m_{\wbeta}(\domain)}=\sum_{k=0}^m\NN{u}^2_{N^k_{\wbeta}(\domain)}$.
Here, $|u|_{H^m(\domain_0)}$ 
signifies the Hilbertian Sobolev semi-norm of integer
order $m$ on $\domain_0$, and $D^{\alpha}$ denotes
the weak partial derivative of order~$\alpha\in \mathbb{N}_0^d$. 
The space $N^m_{\wbeta}(\domain)$ is
the weighted Sobolev space obtained as the closure
of~$C^\infty_0(\domain)$ with respect to the
norm~$\NN{\cdot}_{N^m_{\wbeta}(\domain)}$. 
\begin{remark}\label{rmk:WghtdSpcs}
The weighted spaces are related, for particular ranges of the weight parameters,
to various other corner-weighted scales of Sobolev spaces.
We refer to \cite{BabGuoRegI,GuoBab3dII,GuoBab3dI,CoDaNicMaxw,CoDaNi2012} 
and references there and the discussion in \cite{hpfem} for details.
\end{remark}
With $N^k_{\wbeta}(\domain)$ as defined in \eqref{eq:sn},
for $\delta > 0$ we define the 
\emph{$\wbeta$-weighted $\delta$-Gevrey regular class of functions
with point singularities at $\S$} 
by
\begin{equation}
\label{eq:Abeta} 
\calG^\delta_{\wbeta}(\S;\domain) 
= 
\bigg\{\, u \in\bigcap_{k\ge 0} 
N^k_{\wbeta}(\domain)\,:\,\text{$\exists\, C_u>0$ s.t.}\ 
|u|_{N^k_{\wbeta}(\domain)} 
\le 
C_u^{k+1}(k!)^\delta \ \forall\, k\in\mathbb{N}_0 \,\bigg\}
\;.
\end{equation}

A wide range of partial differential equations is known to admit
singular solutions in weighted Gevrey classes.
For example,
nonlinear Schr\"odinger equations in electron structure calculations
\cite{BHL2012,CCM2010,FHHS2009},
nonlinear parabolic PDEs with critical growth
\cite{HW2018,SGKM1995},
incompressible Euler equations
\cite{Chemin1998},
linear elasticity 
\cite{GB1993},
stationary Stokes 
\cite{GS2006}
and 
stationary, incompressible Navier Stokes 
\cite{MS2020,HMS2023},
see also the references of the cited papers.
We refer to \cite[Sect. 2.2]{hpfem} for more examples
and a more detailed exposition,
which includes the standard example of a 
linear, elliptic, second order PDE with weighted analytic solutions,
in a polygon.
%%%%%%%%%%%%%%%%%%%%%%%%%%%%%%%%%%%%%%%%%%%%%%%%%%%%%%%%%%%%%%%%%%
\subsection{$hp$-Approximation}
\label{sec:hpAppr}
%%%%%%%%%%%%%%%%%%%%%%%%%%%%%%%%%%%%%%%%%%%%%%%%%%%%%%%%%%%%%%%%%%
It is well-known (e.g. \cite{hpfem} and the references there)
that functions $u\in \calG^\delta_{\wbeta}(\S;\domain)$ admit 
approximations from systems of continuous, piecewise polynomial functions
at exponential rates in terms of the number $N$ of degrees of freedom
defining the approximations. In these ``$hp$-'' resp. ``variable mesh and
degree'' approximations, geometric subdivisions toward the singular support 
of $u$ are coupled to an increase in polynomial degree.
We recapitulate from \cite[Sect.~3]{hpfem} the construction 
of these approximations, and the corresponding exponential
approximation rate bounds. 
These comprise corresponding results first
obtained by I.M. Babu\v{s}ka and his coworkers in the analytic case
(where $\delta =1$ in \eqref{eq:Abeta}), 
see e.g. \cite{GuoDiss85,ApproxhpFE,BabGuoCurved},
and will form the basis for corresponding DNN emulation rate bounds.
%%%%%%%%%%%%%%%%%%%%%%%%%%%%%%%%%%%%%%%%%%%%%%%%%%%%%%%%%%%%%%%%%%
\subsubsection{Geometric Meshes}
\label{sec:GeoMes}
%%%%%%%%%%%%%%%%%%%%%%%%%%%%%%%%%%%%%%%%%%%%%%%%%%%%%%%%%%%%%%%%%%
We recall from \cite[Sect.~3.1]{hpfem} 
the notion of \emph{geometric mesh sequences} 
$\MMM_{\kappa,\sigma} = \{ \M^{(\ell)} \}_{\ell \geq 1}$
in $\domain$. 
Such mesh families in $\domain$ constitute an essential 
ingredient in the exponential convergence analysis of
$hp$-approximations \cite{GuoDiss85,ApproxhpFE,BabGuoCurved,Schwabphp,hpfem}.
Specifically, geometric mesh sequences
are sequences of regular, simplicial partitions
of $\domain$ for which there exist two parameters
$\sigma\in (0,1)$ and $\kappa>1$ with the following properties: 
\begin{itemize}
\item[(i)] 
All elements $K\in \M^{(\ell)}$, $\ell = 1,2,\ldots$ 
are uniformly $\kappa$-shape regular, 
i.e. 
there exists a constant $\kappa > 1$ such that
$\sup_{{\ell\in\N}} \kappa_{\rm sh}(\M^{(\ell)}) \leq \kappa$.
\item[(ii)] 
The partitions $\M^{(\ell)}  \in \MMM_{\kappa,\sigma}$
are $\sigma$-geometric, i.e.
for every 
$K\in \M^{(\ell)}: K \cap \S = \emptyset$, $\ell = 1,2, \ldots$
holds
\be\label{eq:DefGeoRatT}
0 < \sigma < 
\frac{{\rm diam}(K)}{{\rm dist}(K,\S)} < \frac{1}{\sigma}
\;.
\ee
\end{itemize}
It was shown in \cite[Prop.~1]{hpfem} that 
$(\kappa,\sigma)$-geometric mesh sequences $\MMM_{\kappa,\sigma}$ in $\domain$ 
have the following geometric properties:
for every $\ell>d$
all elements $K\in \M^{(\ell)}$
can be grouped in {\em mesh layers}:
there exists a partition
\begin{align} \label{eq:MOT}
\M^{(\ell)}
= &\,
\OOO^{(\ell)}\dcup \TTT^{(\ell)}
\;,
\qquad\text{ where }\qquad
\snorm{ \TTT^{(\ell)} } 
	\leq
		C_\TTT({\kappa,\sigma})
,
\end{align}
and, for $k\simeq \ell\log(2)/|\log(\sigma)|$,
there are partitions
\begin{align}
\label{eq:meshdecomp}
\OOO^{(\ell)} 
= 
\LLL^{(\ell)}_1 \dcup \LLL^{(\ell)}_2 \dcup \ldots \dcup \LLL^{(\ell)}_{k}
\;,
\end{align} 
such that 
there exists $c_\TTT>0$ such that for all $\ell$ holds
\begin{equation}\label{eq:termlay}
\S \subset \bigcup_{K\in \TTT^{(\ell)}} K
\;,
\qquad 
{\rm dist}(\S,\OOO^{(\ell)}) \geq c_\TTT \sigma^{k}
\;.
\end{equation}
Furthermore,
there exists a constant $c(\MMM_{\kappa,\sigma}) \geq 1$ with
\be\label{eq:LayerCard}
\forall j=1,\ldots,k :\qquad \snorm{\LLL^{(\ell)}_{j} } \leq c(\MMM_{\kappa,\sigma})
\ee
and such that, for every $j=1,\ldots,k$ and for every $K\in \LLL^{(\ell)}_j$,
\be\label{eq:GeomRat}
0 < \frac{1}{c(\MMM_{\kappa,\sigma})}
\leq
\frac{{\rm diam}(K)}{ \sigma^{j} }
\leq
c(\MMM_{\kappa,\sigma})
\;.
\ee
We will frequently use that 
\eqref{eq:MOT}, \eqref{eq:meshdecomp} and \eqref{eq:LayerCard} 
imply 
\begin{align}
\label{eq:meshcard}
\snorm{\M^{(\ell)}} \simeq k \simeq \ell
.
\end{align}
As shown in \cite[Prop.~1]{hpfem},
for a given polytopal domain $\domain$, 
singular set $\S$ and 
regular, simplicial initial triangulation $\M^{(0)}$, 
\cite[Alg. 1]{hpfem}
provides an explicit construction of 
a $\kappa$-shape-regular and $\sigma$-geometrically 
refined mesh sequence $\MMM_{\kappa,\sigma}$.
%
%%%%%%%%%%%%%%%%%%%%%%%%%%%%%%%%%%%%%%%%%%%%%%%%%%%%%%%%%%%%%%%%%%
\subsubsection{Exponential Convergence}
\label{sec:ExpCnv}
%%%%%%%%%%%%%%%%%%%%%%%%%%%%%%%%%%%%%%%%%%%%%%%%%%%%%%%%%%%%%%%%%%
Based on the geometric mesh sequences $\MMM_{\kappa,\sigma}$
in $\domain$, 
for $u\in \calG^\delta_{\wbeta}(\S;\domain)$,
there exist sequences of continuous,
piecewise polynomial (on $\M^{(\ell)}\in \MMM_{\kappa,\sigma}$) 
functions in $\domain$
approximating $u$ at an exponential rate.
\begin{proposition}[{{\cite[Thm.~1]{hpfem}}}]
\label{prop:eta}
In a bounded polytope 
$\domain\subset \mathbb{R}^d$, $d=2,3$,
with plane sides resp. plane faces,
suppose given a weight vector $\wbeta$ as in~\eqref{eq:beta} 
if $d=3$ or \eqref{eq:beta2} if $d=2$.

Then, for every sequence $\MMM_{\kappa,\sigma}(\S)$ of nested,
regular simplicial meshes in $\domain$
which are $\sigma$-geometrically refined towards $\S$
and which are $\kappa$-shape regular,
there exist continuous projectors 
$\Pi^{p}_{\kappa,\sigma}: 
  N^2_{\wbeta}(\domain) \to S^p(\M^{(\ell)})$
with $\ell\simeq p^{1/\delta}$
and, 
for every $u\in \calG^\delta_{\wbeta}(\S;\domain)$ 
there exist constants $b,C>0$
(depending on $\kappa$, $C_u$, $d_u$ in \eqref{eq:Abeta} and on $\sigma$)
such that there holds the error bound
\begin{equation}
\label{eq:eta}
\NN{u - \Pi^{p}_{\kappa,\sigma} u }_{H^1(\domain)} 
\leq C 
\begin{cases} \;\exp(-bN^{\frac{1}{1+\delta d}}) &\delta\geq 1,
\\
\; 
\left(\Gamma\left(N^{\frac{1}{1+\delta d}}\right)\right)^{-b(1-\delta)} &0<\delta<1.
\end{cases}
\end{equation}
Here,
$$
N={\rm dim}(S^p(\M^{(\ell)})) 
\simeq \snorm{\M^{(\ell)}} p^d
\simeq \ell p^{d} 
\simeq p^{d + {1/\delta}}.
$$
If, additionally, $u|_{\partial \domain} = 0$, 
then
$( \Pi^{p}_{\kappa,\sigma} u)|_{\partial \domain} = 0$ 
and \eqref{eq:eta} holds.
\end{proposition}
% 
%%%%%%%%%%%%%%%%%%%%%%%%%%%%%%%%%%%%%%%%%%%%%%%%%%%%%%%%%%%%%%%%%%
\section{Neural Networks}
\label{sec:nnbasics}
%%%%%%%%%%%%%%%%%%%%%%%%%%%%%%%%%%%%%%%%%%%%%%%%%%%%%%%%%%%%%%%%%%
%
The continuous, piecewise polynomial approximations 
in $S^p(\M^{(\ell)})$ 
of $u\in \calG^\delta_{\wbeta}(\S;\domain)$
imply the existence of deep neural networks with corresponding
exponential approximation rates. 
In Sect. \ref{sec:nnbasics}--\ref{sec:hofem},
we develop a rigorous statement of this fact. 

Sect.~\ref{sec:nndef} introduces basic notation and NN terminology,
from \cite{LODSZ22_991} and the references there. 
Sect.~\ref{sec:cpwlnn} recalls, from \cite{LODSZ22_991},
a result on the exact representation of 
continuous, piecewise linear (CPwL)
functions on the geometric partitions
$\M^{(\ell)}$ in Prop.~\ref{prop:eta}, 
by means of a ReLU-activated NN in $\domain$.

The main results on NN approximation of $u\in \calG^\delta_{\wbeta}(\S;\domain)$
are then developed in Sect.~\ref{sec:efficienthofem} ahead.
%%%%%%%%%%%%%%%%%%%%%%%%%%%%%%%%%%%%%%%%%%%%%%%%%%%%%%%%%%%%%%%%%
\subsection{Neural Network Definitions}
\label{sec:nndef}
%%%%%%%%%%%%%%%%%%%%%%%%%%%%%%%%%%%%%%%%%%%%%%%%%%%%%%%%%%%%%%%%%%

\begin{definition}[Neural Network {{ \cite[Sect. 2.1]{LODSZ22_991} }}]
\label{def:nn} For $d,L\in\N$, a \emph{neural
  network $\Phi$} with input dimension $d \geq 1$ and number of
  layers $L\geq 1$, comprises a finite collection of activation
  functions\footnote{No activation is applied in the output layer $L$.
  We introduce $\varrho_{L}$ only for consistency of notation,
  and define it to be equal to the identity function.}
  $\bsvarrho =
  \{\varrho_{\ell}\}_{\ell=1}^L$
  and a finite sequence of matrix-vector tuples, i.e.
  \begin{align*}
    \Phi = ((A_1,b_1,\varrho_1),(A_2,b_2,\varrho_2),\ldots,(A_L,b_L,\varrho_L)).
  \end{align*}

  For $N_0 := d$ and \emph{numbers of neurons $N_1,\ldots,N_L\in\N$
  per layer}, for all $\ell=1,\ldots, L$ it holds that
  $A_\ell\in\bbR^{N_\ell \times N_{\ell-1} }$ and
  $b_\ell\in\bbR^{N_\ell}$, and that $\varrho_\ell$ is a list of
  length $N_\ell$ of \emph{activation functions}
  $(\varrho_\ell)_i: \R\to\R$, $i=1,\ldots,N_\ell$, acting on node $i$
  in layer $\ell$.

  The \emph{realization} of $\Phi$ as a map is
  the function
  \begin{align*}
    \realiz{\Phi}: \R^d\to\R^{N_L} : x \to x_L,
  \end{align*}
  where
  \begin{align*}
    x_0 & := x,
    \\
    x_\ell & := \varrho_\ell( A_\ell x_{\ell-1} + b_\ell ),
             \qquad\text{ for } \ell=1,\ldots,L.
  \end{align*}
  Here, for
  $\ell=1,\ldots,L$, 
  the list of activation functions
  $\varrho_\ell$ of length $N_\ell$ is effected
  component\-wise:
  for $y = (y_1,\ldots,y_{N_\ell})\in\bbR^{N_\ell}$ we
  denote
  $\varrho_\ell(y) = ( (\varrho_\ell)_1(y_1), \ldots,
  (\varrho_\ell)_{N_\ell}(y_{N_\ell}) )$.
  I.e., $(\varrho_\ell)_i$ is
  the activation function applied in position $i$ of layer $\ell$.

  We call the layers indexed by $\ell=1,\ldots,L-1$ \emph{hidden
    layers}, in those layers activation functions are applied. 
  We fix the activation function in the last layer of the NN to be the identity, 
  i.e., $\varrho_L := \Id_{\R^{N_L}}$.

  We refer to $\depth(\Phi) := L$ as the \emph{depth} of $\Phi$.  
  For
  $\ell=1,\ldots,L$ we denote by
  $\size_\ell(\Phi) := \norm[0]{A_\ell} + \norm[0]{b_\ell} $ the
  \emph{size of layer $\ell$}, which is the number of nonzero
  components in the weight matrix $A_\ell$ and the bias vector
  $b_\ell$, and call $\size(\Phi) := \sum_{\ell=1}^L \size_\ell(\Phi)$
  the \emph{size} of $\Phi$.
  Furthermore,
  we call $d$ and $N_L$ the
  \emph{input dimension} and the \emph{output dimension}, and denote
  by $\sizefirst(\Phi) := \size_1(\Phi)$ and
  $\sizelast(\Phi) := \size_L(\Phi)$ the size of the first and the
  last layer, respectively.
\end{definition}
Our networks will use two different activation functions.  
Firstly, we use the \emph{Rectified Linear Unit} (\emph{\ReLUnospace}) activation
\begin{equation}
\label{eq:reludef}
    \relu(x) = \max\{ 0, x \}.
\end{equation}
Networks which only contain \ReLU activations realize continuous,
piecewise linear functions. 
By \emph{\ReLU NNs} we refer to NNs which
only have \ReLU activations, including networks of depth $1$, which do
not have hidden layers and realize affine transformations.
Secondly, for the emulation of high-order finite element methods, 
we use the \ReLUtwo activation
\begin{align}
  \label{eq:relu2def}
  \relu^2(x) = \max\{ 0, x \}^2.
\end{align}

\begin{remark}
\label{rem:numberofneurons}
Upper bounds on the network size
also provide upper bounds on the number of neurons.
Without loss of generality, 
the network size is bounded from below by $\sum_{j=1}^L N_j$,
see \cite[Lem. G.1]{PV2018}.
There, it was proved formally that
each neuron to which no nonzero weights are associated can be omitted.
\end{remark}

In the following sections, we will construct NNs from smaller networks
using a 
\emph{calculus of NNs}, which we now recall from \cite{PV2018}.
The results cited from \cite{PV2018} were derived for
NNs which only use the \ReLU activation function, but they also hold
for networks with multiple activation functions without modification.

\begin{proposition}[Parallelization of NNs {{\cite[Def. 2.7]{PV2018}}}]
  \label{prop:parallel}
  For $d,L\in\N$ let $\Phi^1 = $ \linebreak
  $ \left
    ((A^{(1)}_1,b^{(1)}_1,\varrho^{(1)}_1),\ldots,(A^{(1)}_L,b^{(1)}_L,\varrho^{(1)}_L)\right
  ) $ and
  $\Phi^2 = \left
    ((A^{(2)}_1,b^{(2)}_1,\varrho^{(2)}_1),\ldots,(A^{(2)}_L,b^{(2)}_L,\varrho^{(2)}_L)\right
  ) $ be two NNs with input dimension $d$ and depth $L$.  Let the
  \emph{parallelization} $\Parallel{\Phi^1,\Phi^2}$ of $\Phi^1$ and
  $\Phi^2$ be defined by
  \begin{align*}
    \Parallel{\Phi^1,\Phi^2} := &\, ((A_1,b_1,\varrho_1),\ldots,(A_L,b_L,\varrho_L)),
    &&
    \\
    A_1 = &\, \begin{pmatrix} A^{(1)}_1 \\ A^{(2)}_1 \end{pmatrix},
    \quad
    A_\ell = \begin{pmatrix} A^{(1)}_\ell & 0\\0&A^{(2)}_\ell \end{pmatrix},
    &&
       \text{ for } \ell = 2,\ldots L,
    \\
    b_\ell = &\, \begin{pmatrix} b^{(1)}_\ell \\ b^{(2)}_\ell \end{pmatrix},
    \quad
    \varrho_\ell = \begin{pmatrix} \varrho^{(1)}_\ell \\ \varrho^{(2)}_\ell \end{pmatrix},
                                &&
                                   \text{ for } \ell = 1,\ldots L.
  \end{align*}

  Then,
  \begin{align*}
    \realiz{\Parallel{\Phi^1,\Phi^2}} (x)
    = &\, ( \realiz{\Phi^1}(x), \realiz{\Phi^2}(x) ),
        \quad
        \text{ for all } x\in\R^d,
    \\
    \depth(\Parallel{\Phi^1,\Phi^2}) = L, &
                                            \qquad
                                            \size(\Parallel{\Phi^1,\Phi^2}) = \size(\Phi^1) + \size(\Phi^2)
                                                  .
  \end{align*}

\end{proposition}
The parallelization of more than two NNs is handled by repeated
application of Prop. \ref{prop:parallel}.

We will also use the parallelization of networks
which do not have the same inputs.
\begin{proposition}[Full Parallelization of NNs {{\cite[Setting 5.2]{EGJS2021}}}] 
\label{prop:parallSep}
For $L \in \N$ let
\linebreak
$ \Phi^1 = \Big(
	(A^{(1)}_1,b^{(1)}_1),\ldots,(A^{(1)}_L,b^{(1)}_L)
    \Big)$ 
and
$\Phi^2 = \Big(
	(A^{(2)}_1,b^{(2)}_1), \ldots,(A^{(2)}_L,b^{(2)}_L)
	\Big)$
be two NNs with the same depth $L$,
with input dimensions $N^1_0=d_1$ and $N^2_0=d_2$, respectively. 
Let the \emph{full parallelization} of $\Phi^1$ and $\Phi^2$
be defined by
\begin{align*}
\FParallel{\Phi^1,\Phi^2} := &\, ((A_1,b_1,\varrho_1),\ldots,(A_L,b_L,\varrho_L)),
	\\
    A_\ell = &\, \begin{pmatrix} A^{(1)}_\ell & 0\\0&A^{(2)}_\ell \end{pmatrix},
	\quad
    b_\ell = \begin{pmatrix} b^{(1)}_\ell \\ b^{(2)}_\ell \end{pmatrix},
    \quad
    \varrho_\ell = \begin{pmatrix} \varrho^{(1)}_\ell \\ \varrho^{(2)}_\ell \end{pmatrix},
    \quad
    \text{ for } \ell = 1,\ldots L
	.
\end{align*}

Then, it has $d = d_1+d_2$-dimensional input, 
depth $L$ and size 
$\size(\FParallel{\Phi^1,\Phi^2}) 
	= \size(\Phi^1) + \size(\Phi^2)$.
It satisfies
for all $x = (x_1,x_2) \in \R^d$ with $x_i \in \R^{d_i}, i = 1,2$ 
\begin{align*} 
\realiz{\FParallel{\Phi^1,\Phi^2}} (x_1,x_2) 
	= &\, \left(\realiz{\Phi^1}(x_1), \realiz{\Phi^2}(x_2)\right)
.
\end{align*}
\end{proposition}

Next, we define the concatenation of two NNs.
\begin{definition}[Concatenation of NNs {{\cite[Def. 2.2]{PV2018}}}]
\label{def:pvconc}
For $L^{(1)}, L^{(2)} \in\N$,
let 
$ \Phi^1 = $ \linebreak
$ \Big(
	(A^{(1)}_1,b^{(1)}_1,\varrho^{(1)}_1),\ldots,$ $(A^{(1)}_{L^{(1)}},b^{(1)}_{L^{(1)}},\varrho^{(1)}_{L^{(1)}})
	\Big) $ and
$\Phi^2 = \Big(
	(A^{(2)}_1,b^{(2)}_1,\varrho^{(2)}_1),$
	$\ldots,(A^{(2)}_{L^{(2)}},b^{(2)}_{L^{(2)}},\varrho^{(2)}_{L^{(2)}})
	\Big)$
be two NNs 
such that the input dimension of $\Phi^1$
equals the output dimension of $\Phi^2$.
Then, the \emph{concatenation} of $\Phi^1$ and $\Phi^2$ 
is the NN of depth $L := L^{(1)} + L^{(2)} -1$ defined as
\begin{align*}
    \Phi^1 \bullet \Phi^2 := &\, ((A_1,b_1,\varrho_1),\ldots,(A_L,b_L,\varrho_L)),
    \\
    (A_\ell,b_\ell,\varrho_\ell) = &\, (A^{(2)}_\ell,b^{(2)}_\ell,\varrho^{(2)}_\ell),
	\qquad
    \text{ for }\ell=1,\ldots,L^{(2)}-1,
    \\
    A_{L^{(2)}} = &\, A^{(1)}_1 A^{(2)}_{L^{(2)}},
    \qquad
    b_{L^{(2)}} = A^{(1)}_1 b^{(2)}_{L^{(2)}} + b^{(1)}_1,
    \qquad
    \varrho_{L^{(2)}} = \varrho^{(1)}_1,
    \\
    (A_\ell,b_\ell,\varrho_\ell) 
    	= &\, (A^{(1)}_{\ell-L^{(2)}+1},b^{(1)}_{\ell-L^{(2)}+1}, \varrho^{(1)}_{\ell-L^{(2)}+1}),
	\qquad
	\text{ for }\ell=L^{(2)}+1,\ldots,L^{(1)}+L^{(2)}-1.
\end{align*}
\end{definition}
It follows immediately from this definition that
$\realiz{ \Phi^1 \bullet \Phi^2 } = \realiz{ \Phi^1 } \circ \realiz{ \Phi^2 }$.

In addition, we define the sparse concatenation of two NNs based on the \ReLUtwo activation,
which also exactly realizes the composition of the realizations of the two networks.
The sparse concatenation is a construction 
which allows to bound the size of the concatenation 
as a constant times the sum of the sizes of the individual networks. 
This bound does not hold if we 
combine the affine transformation of the output layer of $ \Phi^2 $ 
with the affine transformation of the input layer of $ \Phi^1 $,
as we did in Def. \ref{def:pvconc}.

\begin{proposition}[Sparse Concatenation of NNs based on \ReLUtwo {{\cite[Prop. 2.4]{OSZ19_2771}}}]
\label{prop:relutwoconcat}
For $L^{(1)},L^{(2)}\in\N$, let $\Phi^1$ and $\Phi^2$
be two NNs with depths $L^{(1)}$ and $L^{(2)}$, respectively, such
that $N^{(2)}_{L^{(2)}} = N^{(1)}_0$, i.e. the output dimension of
$\Phi^2$ equals the input dimension of $\Phi^1$.
Then, 
there exists a NN $\Phi^1\sconc_{\relu^2} \Phi^2$ 
of depth $L := L^{(1)} + L^{(2)}$ 
to which we shall refer as \emph{sparse \ReLUtwonospace-based concatenation} 
of $\Phi^1$ and $\Phi^2$ which has \ReLUtwo activations, 
as well as those from $\Phi^1$ and $\Phi^2$.
It satisfies
\begin{align*}
\realiz{\Phi^1 \sconc_{\relu^2} \Phi^2} = &\, \realiz{\Phi^1} \circ \realiz{\Phi^2},
	\qquad
	\depth(\Phi^1 \sconc_{\relu^2} \Phi^2) = \, L^{(1)} + L^{(2)},
	\\
\size(\Phi^1 \sconc_{\relu^2} \Phi^2) 
	\leq &\, \size(\Phi^1) + 4 \sizefirst(\Phi^1) + 7 \sizelast(\Phi^2) + \size(\Phi^2)
    \leq 5 \size(\Phi^1) + 8 \size(\Phi^2)
	.
\end{align*}
\end{proposition}

Prop. \ref{prop:parallel} and \ref{prop:parallSep}
only apply to networks of equal depth.  
To parallelize two networks of unequal depth, 
the shallower of the two
can be concatenated with a network that
emulates the identity using 
Prop.~\ref{prop:relutwoconcat}.

Next, we recall the exact emulation of the identity and of products by \ReLUtwo NNs.

\begin{proposition}[\ReLUtwo NN emulation of $\Id_{\R^d}$ {{\cite[Prop. 2.3]{OSZ19_2771}}}]
\label{prop:relutwoidnn}
For all $d\in\N$ and $L\in\N$,
there exists a \ReLUtwo NN $\Phi^{\Id,\relu^2}_{d,L}$ such that 
$\realiz{\Phi^{\Id,\relu^2}_{d,L}}(x) = x$ for all $x\in\R^d$.
Its NN depth and size satisfy
$\depth(\Phi^{\Id,\relu^2}_{d,L}) = L$ 
and 
$\size(\Phi^{\Id,\relu^2}_{d,L}) \leq C d L$,
for $C>0$ independent of $d$ and $L$.
\end{proposition}

\begin{proposition}[\ReLUtwo NN emulation of products]
\label{prop:nnprod}
For all $d\in\N$, $d\geq2$,
there exists a \ReLUtwo NN $\Phi^{\mathrm{prod}}_d$ such that 
$\realiz{\Phi^{\mathrm{prod}}_d}(x_1,\ldots,x_d) = \prod_{j=1}^d x_j$ 
for all $x_1,\ldots,x_d\in\R$.
Its NN depth and size satisfy
$\depth(\Phi^{\mathrm{prod}}_d) \leq C \ceil{\log_2(d)}$
and 
$\size(\Phi^{\mathrm{prod}}_d) \leq C d$,
for $C>0$ independent of $d$.
\end{proposition}

\begin{proof}
This proof is given in three steps.
In Step 1, 
we construct $\Phi^{\mathrm{prod}}_8$
and analyze its depth and size.
In Step 2,
we construct and analyze
$\Phi^{\mathrm{prod}}_d$
for $d>8$ satisfying $d\in 8^\N$
as an octree of $\Phi^{\mathrm{prod}}_8$ NNs.
Here, $8^\N = \{ 8^k : k\in\N \}$ denotes the set of positive integer powers of $8$.
The reason for considering an octree of $\Phi^{\mathrm{prod}}_8$ NNs,
is that we need $8$ to be bigger than the number $5$ from the estimate
$\size(\Phi^1 \sconc_{\relu^2} \Phi^2) 
	\leq 5 \size(\Phi^1) + 8 \size(\Phi^2)$
from Prop. \ref{prop:relutwoconcat}.
Finally, Step 3
considers $d\geq 2$ satisfying $d\notin 8^\N$.

\textbf{Step 1.}
For the product of two numbers,
$\Phi^{\mathrm{prod}}_2$ is given in \cite[Lem. 2.1]{LTY2020RePU}.
It satisfies 
$\realiz{ \Phi^{\mathrm{prod}}_2 }(x,y) = xy$ for all $x,y\in\R$,
$\depth( \Phi^{\mathrm{prod}}_2 ) = 2$, 
$N_0 = 2$, $N_1=4$, $N_2=1$ 
and
$\size( \Phi^{\mathrm{prod}}_2 ) = 12$.
For $d=8$, we define
\begin{align*}
\Phi^{\mathrm{prod}}_8 
	:= \Phi^{\mathrm{prod}}_2 \sconc_{\relu^2} 
		\FParallel{ \Phi^{\mathrm{prod}}_2, \Phi^{\mathrm{prod}}_2 } \sconc_{\relu^2} 
		\FParallel{ \Phi^{\mathrm{prod}}_2, \Phi^{\mathrm{prod}}_2, 
			\Phi^{\mathrm{prod}}_2, \Phi^{\mathrm{prod}}_2 }
.
\end{align*}
By Prop. \ref{prop:relutwoconcat} it satisfies 
$\depth( \Phi^{\mathrm{prod}}_8 ) = 2+2+2 = 6$,
and we denote its size by 
$M_8 := \size( \Phi^{\mathrm{prod}}_8 )$.
The network $\Phi^{\mathrm{prod}}_2$ exactly emulates the product of two numbers,
and no errors are incurred in the sparse concatenation $\sconc_{\relu^2}$,
thus 
$\realiz{ \Phi^{\mathrm{prod}}_8 } : \R^8\to\R: 
	(x_1,\ldots,x_8) \mapsto \prod_{i=1}^8 x_i$.

\textbf{Step 2.}
For $d>8$ satisfying $d\in 8^\N$, 
we define 
\begin{align*}
\Phi^{\mathrm{prod}}_d 
	:= \Phi^{\mathrm{prod}}_{d/8} \sconc_{\relu^2} 
		\FParallel{ \Phi^{\mathrm{prod}}_8, \ldots, \Phi^{\mathrm{prod}}_8 }
,
\end{align*}
where the full parallelization contains $d/8$ product networks.
Based on Prop. \ref{prop:relutwoconcat}
and the depth and size bounds for the $d=8$ product network, 
we inductively obtain that 
for all $d\in 8^\N$ satisfying $d>8$ holds
$\depth( \Phi^{\mathrm{prod}}_d ) = 2 \ceil{\log_2(d)}$
and
$\size( \Phi^{\mathrm{prod}}_d ) \leq 8 d M_8$:
\begin{align*}
\depth( \Phi^{\mathrm{prod}}_d )
	= &\, \depth( \Phi^{\mathrm{prod}}_{d/8} ) + \depth( \Phi^{\mathrm{prod}}_8 )
	= 2 \ceil{\log_2(d/8)} + 6
	= 2 \ceil{\log_2(d)}
	,
	\\
\size( \Phi^{\mathrm{prod}}_d )
	\leq &\, 5 \size( \Phi^{\mathrm{prod}}_{d/8} ) + 8 (d/8) \size( \Phi^{\mathrm{prod}}_8 )
	\\
	\leq &\, 5 \cdot 8 ( d/8 ) M_8 + 8 (d/8) M_8 
	= ( 5 \cdot 8 + 8 ) (d/8) M_8
	\leq ( 8 \cdot 8 ) (d/8) M_8
	= 8 d M_8
	.
\end{align*}
The fact that
$\realiz{ \Phi^{\mathrm{prod}}_d } : \R^d\to\R: 
	(x_1,\ldots,x_d) \mapsto \prod_{i=1}^d x_i$
follows from exactness of 
$\realiz{ \Phi^{\mathrm{prod}}_{d/8} }$ 
and
$\realiz{ \Phi^{\mathrm{prod}}_8 }$
and the fact that no error is incurred in the sparse concatenation $\sconc_{\relu^2}$.

\textbf{Step 3.}
Finally, for $d\geq2$ satisfying $d\notin 8^\N$,
let $\tilde{d} = \min\{ 8^k : k\in\N, 8^k\geq d \} \in 8^\N$.
By definition, it holds that $\tilde{d} \leq 8d$.
We obtain
$\Phi^{\mathrm{prod}}_d$ 
from 
$\Phi^{\mathrm{prod}}_{\tilde{d}}$
by setting the last $\tilde{d}-d$ inputs of the NN to $1$,
through the biases in the first layer.
This 
gives the desired realization and
increases the network size by at most $\tilde{d} - d$.
We thus obtain that
\begin{align*}
\depth( \Phi^{\mathrm{prod}}_d )
	= &\, \depth( \Phi^{\mathrm{prod}}_{\tilde{d}} )
	= 2 \ceil{\log_2(\tilde{d})}
	\leq 2 \ceil{\log_2(d)} + 6
	\leq C \ceil{\log_2(d)}
	,
	\\
\size( \Phi^{\mathrm{prod}}_d )
	\leq &\, \size( \Phi^{\mathrm{prod}}_{\tilde{d}} ) + (\tilde{d} - d)
	\leq 8 \tilde{d} M_8 + \tilde{d}
	\leq C d.
\end{align*}
\end{proof}

Finally, we recall the exact emulation of polynomials by \ReLUtwo networks.
We need the result from \cite{OSZ19_2771}, 
which holds for multivariate polynomials,
only for the special case of univariate polynomials.

\begin{proposition}[\ReLUtwo emulation of univariate polynomials {{\cite[Prop. 2.14]{OSZ19_2771}}}]
\label{prop:univarpol}
For all $p\in\N$ and $w\in\bbP_p$,
there exists a \ReLUtwo NN $\Phi^w$ such that 
$\realiz{\Phi^w} = w$,
$\depth(\Phi^w) \leq C \log_2(p+1)$ 
and 
$\size(\Phi^w) \leq C p$,
for $C>0$ independent of $p$ and $w$.
\end{proposition}
%%%%%%%%%%%%%%%%%%%%%%%%%%%%%%%%%%%%%%%%%%%%%%%%%%%%%%%%%%%%%%%%%%
\subsection{NN Emulations of Continuous, Piecewise Linear Functions on Simplicial Meshes}
\label{sec:cpwlnn}
%%%%%%%%%%%%%%%%%%%%%%%%%%%%%%%%%%%%%%%%%%%%%%%%%%%%%%%%%%%%%%%%%%
We recall from \cite{LODSZ22_991}
\emph{exact NN emulations} 
of CPwL functions on regular, simplicial partitions $\calT$ 
of polytopal domains $\domain \subset \bbR^d$.

We have by \cite[Sect. 5]{LODSZ22_991} a vector space of NNs
$\calNN(CPwL;\calT,\domain) 
	= \{ \Phi^{CPwL,v} : v\in\So(\calT,\domain) \}$
such that the realization of each NN $\Phi^{CPwL,v}$ equals $v$ everywhere in $\overline\domain$.

\begin{proposition}[{{ \cite[Prop. 5.7]{LODSZ22_991} }}]
\label{prop:CPLbasisnet}
Let $\domain\subset\R^d$, $d\geq2$, be a bounded, polytopal domain.
For every regular, simplicial triangulation $\calT$ of $\domain$,
there exists a NN $\Phi^{CPwL} :=\Phi^{CPwL(\calT,\domain)}$ with
only \ReLU activations, which in parallel emulates the shape
functions $\{ \theta^{\So}_i \}_{i\in \calV}$,
which are defined by $\theta^{\So}_i(i) = 1$ 
and $\theta^{\So}_i(j) = 0$ for all other $j\in\calV$.
That is, 
$ \realiz{\Phi^{CPwL}}\colon \domain \to \R^{|\calV|} $
satisfies
\begin{equation*}
\realiz{\Phi^{CPwL}}_i(x)
= \, \theta^{\So}_i(x)
\quad\text{ for all } x\in\domain
\text{ and all } i\in\calV.
\end{equation*}
There exists $C>0$ independent of $d$ and $\calT$ such that
\begin{align*}
\depth(\Phi^{CPwL})
\leq &\, 8 + \log_2( \mathfrak{s}(\calV) ) + \log_2(d+1)
,
\\
\size(\Phi^{CPwL})
\leq & C \snorm{\calV} \log_2( \mathfrak{s}(\calV) ) + C d^2 \sum_{i\in\calV} s(i)
\leq C d^2 \mathfrak{s}(\calV) \dim( \So(\calT,\domain)).
\end{align*}

For all
$v = \sum_{i\in\calV} v_i \theta^{\So}_i \in\So(\calT,\domain)$,
there exists a NN $\Phi^{CPwL,v} := \Phi^{CPwL(\calT,\domain),v} $
with only \ReLU activations, such that for a constant $C>0$
independent of $d$ and $\calT$
\begin{align*}
\realiz{\Phi^{CPwL,v}}(x)
= &\, v(x)
\quad\text{ for all } x\in\domain
,
\\
\depth(\Phi^{CPwL,v})
\leq &\, 8 + \log_2( \mathfrak{s}(\calV) ) + \log_2(d+1)
,
\\
\size(\Phi^{CPwL,v})
\leq &\, C \snorm{\calV} \log_2( \mathfrak{s}(\calV) ) + C d^2 \sum_{i\in\calV} s(i)
\leq C d^2 \mathfrak{s}(\calV) \dim( \So(\calT,\domain)).
\end{align*}

The layer dimensions and the lists of activation functions of
$\Phi^{CPwL}$ and $\Phi^{CPwL,v}$ are independent of $v$ and only
depend on $\calT$ through $ \{s(i)\}_{i\in\calV} $ and
$ \snorm{\calV} = \dim(\So(\calT,\domain)) $.

The set
$\calNN(CPwL;\calT,\domain) := \{ \Phi^{CPwL,v} : v \in
\So(\calT,\domain) \}$ together with the linear operation
$\Phi^{CPwL,v} \widehat{+} \lambda\Phi^{CPwL,w} := \Phi^{CPwL,v+\lambda w}$
for all $v,w\in \So(\calT,\domain)$ and all $\lambda\in\R$ is a
vector space.
The realization map
$\realiz{\cdot}: \calNN(CPwL;\calT,\domain) \to \So(\calT,\domain)$
is a linear isomorphism.
\end{proposition}

For a proof, we refer to \cite[Sect. 5]{LODSZ22_991}.

\begin{remark}[{{ \cite[Rem. 5.2]{LODSZ22_991} }}]
\label{rem:sizebound}
It was shown in \cite[Rem. 5.2]{LODSZ22_991}
that
$\snorm{\calV} \leq \sum_{i \in \calV} s(i) \leq c(\calV,d) \snorm\calT$,
where $c(\calV,d) = d+1$ is the number of vertices of a $d$-simplex.
We obtain this inequality by observing that
$s(i) \geq 1$ and that
each element $K \in \calT$ contributes $+1$ to $c(\calV,d)$ terms $s(i)$.  
Therefore, we also have the bound
$ \size(\Phi^{CPwL}) 
	\leq C \snorm{\calV} \log_2( \mathfrak{s}(\calV) ) + Cd^2 c(\calV,d) \snorm\calT
	\leq Cd^2 c(\calV,d) \log_2( \mathfrak{s}(\calV) ) \snorm\calT$.
The same bound holds for $\size(\Phi^{CPwL,v})$.
\end{remark}

%%%%%%%%%%%%%%%%%%%%%%%%%%%%%%%%%%%%%%%%%%%%%%%%%%%%%%%%%%%%%%%%%%
\section{NN Emulation of High-Order, Lagrangean Finite Element Spaces}
\label{sec:hofem}
%%%%%%%%%%%%%%%%%%%%%%%%%%%%%%%%%%%%%%%%%%%%%%%%%%%%%%%%%%%%%%%%%%

We state and prove the main result of this paper: 
a \emph{mathematically exact emulation of Lagrangean FE spaces of any order} 
$p\geq 1$ 
on general regular, simplicial triangulations 
in a polytopal domain $\domain\subset \R^d$, 
in space dimension $d\geq 2$, with \ReLU and \ReLUtwo activations, 
and with both the number of neurons and
the DNN size (i.e., the number of nonzero NN parameters, called weights and biases)
in the NN required to realize the emulation being bounded by a 
constant (which depends on the domain $\domain$ and on the mesh connectivity) 
times the dimension of the corresponding FE space.
This result comprises in particular the 
$hp$-Finite Element spaces $S^p(\M^{(\ell)})$ 
on families $\MMM_{\kappa,\sigma}(\S)$ of 
nested, regular simplicial partitions 
$\M^{(\ell)}$ in $\domain$
which are $\sigma$-geometrically refined towards the set $\S$ of singular
support points and which are $\kappa$-shape regular. 
For these DNNs, 
the exponential approximation rate result \eqref{eq:eta} 
in Prop.~\ref{prop:eta} holds. 
This implies exponential rates for the DNN approximation error
for various DNN-based PDE approximation methods.

%%%%%%%%%%%%%%%%%%%%%%%%%%%%%%%%%%%%%%%%%%%%%%%%%%%%%%%%%%%%%%%%%%
\subsection{Existing Results}
\label{sec:previoushofem}
%%%%%%%%%%%%%%%%%%%%%%%%%%%%%%%%%%%%%%%%%%%%%%%%%%%%%%%%%%%%%%%%%%
The exact \ReLU NN emulation of continuous, piecewise linear functions
on general regular, simplicial partitions of polytopal domains
$\domain \subset \R^d$, $d\in\N$,
was achieved in \cite[Prop. 5.7]{LODSZ22_991}, 
as stated above in Prop. \ref{prop:CPLbasisnet}.
In \cite{HX2023}, it was observed that on each simplex $K\in\calT$,
each of the barycentric coordinates equals one of the 
``hat'' basis functions
$\theta^{\So}_q \in \So(\calT,\domain)$ for a vertex $q$ of $K$.
For $q\in\calV$,
these are defined by $\theta^{\So}_q(q) = 1$ 
and $\theta^{\So}_q(q') = 0$ for all other vertices $q'\in\calV$.
Based on an expression of the 
local polynomial space $\bbP_p(K)$ for $p\in\N$ 
in terms barycentric coordinates,
an explicit formula in terms of these hat functions 
for a global basis of high-order finite elements
was given in \cite{HX2023}.
We recall it in Prop. \ref{prop:hofemrepres} below.
To state the result,
on each subsimplex $K'$ of $\calT$ of 
dimension $m \in\{0,\ldots,d\}$
with vertices $a_0,\ldots,a_{m}$,
for all $p\in\N$
we define
the set of interpolation points 
\begin{align*}
\calN_p(K') := &\, \left \{ \sum_{k=0}^{m} \alpha_k a_k / p : 
	\alpha_k \in \N_0 \text{ and } \sum_{k=0}^{m} \alpha_k = p \right \}
.
\end{align*}
As nodal basis for $\bbP_p(K')$ we 
consider the Lagrange polynomials $\{ v_i \}_{i\in \calN_p(K')}$
defined by
$v_i(i) = 1$ and $v_i(j) = 0$ for all other $j\in \calN_p(K')$. 
We denote the set of all interpolation points by 
$\calN := \calN_p(\calT) := \cup_{K\in\calT} \calN_p(K)$
and
consider the global Lagrangean basis functions $\{ \theta^{\Sop{p}}_i \}_{i\in \calN}$
of $\Sop{p}(\calT,\domain)$
defined by 
$\theta^{\Sop{p}}_i(i) = 1$ and $\theta^{\Sop{p}}_i(j) = 0$ for all other $j\in \calN$.

\begin{proposition}[{{\cite[Thm. 3.2]{HX2023}}}]
\label{prop:hofemrepres}
For all $p\in\N$ and a polytopal domain $\domain \subset \R^d$, for $d\in\N$,
let $\calT$ be a regular, simplicial partition of $\domain$.
For each interpolation node $i\in\calN$,
if $i$ is a vertex of $\calT$ 
let $K' = i$ and ${m} = 0$,
and else let $K'$ be the subsimplex of $\calT$ satisfying $i\in\interior{K'}$
and let $m$ denote the dimension of $K'$.
Denoting the vertices of $K'$ by $a_0,\ldots,a_{m}$,
let 
$\alpha_0,\ldots,\alpha_{m} \in \N_0$ 
be such that 
$i = \sum_{k=0}^{m} \alpha_k a_k / p$.

Then, 
\begin{align}
\label{eq:hofemrepres}
\theta^{\Sop{p}}_i(x)
	= &\, 
		\prod_{k=0}^{m} \tfrac{1}{\alpha_k!} 
		\prod_{j=0}^{\alpha_k-1} \left ( p \theta^{\So}_{a_k}(x) - j \right )
	=
		\prod_{k=0}^{m} 
		\prod_{j=0}^{\alpha_k-1} 
			\left ( \tfrac{p}{j+1} \theta^{\So}_{a_k}(x) - \tfrac{j}{j+1} \right )
.
\end{align}
\end{proposition}

From this, the following \ReLU NN emulation rate result is obtained in \cite{HX2023}.

\begin{proposition}[{{\cite[Thm. 4.2]{HX2023}}}]
\label{prop:hxreluhofem}
For all $p\in\N$ and a polytopal domain $\domain \subset \R^d$, for $d\in\N$,
let $\calT$ be a regular, simplicial partition of $\domain$.

For all $v\in\Sop{p}(\calT,\domain)$,
there exists a NN $\Phi_{v}$
with \ReLU and \ReLUtwo activation,
such that 
$\realiz{\Phi_v} = v$.
The depth $L = \depth(\Phi_v) $ 
and the numbers of neurons $N_1,\ldots,N_{L-1}$ in the hidden layers satisfy
for some $C>0$ independent of $\calT$
\begin{align*}
\depth(\Phi_v) 
	\leq &\, \left ( \ceil{ \log_2(\mathfrak{s}(\calV)) } 
		+ \ceil{ \log_2(d+1) } 
		+ \ceil{ \log_2(p) } 
		+ 7 \right )
		\snorm{\calN}
	,
	\\
\max_{\ell=1}^{L-1} N_\ell 
	\leq &\, C \max\{ p, (d+1)\min\{d+1 , p+1\}\mathfrak{s}(\calV) \} 
.
\end{align*}

\end{proposition}

\begin{remark}
\label{rem:hx2023discussion}
The proof of \cite[Thm. 4.2]{HX2023}
is based on \cite[Lem. 4.2]{HX2023},
in which single Lagrangean basis functions are emulated.
For each $p\in\N$ and each interpolation point $i\in\calN$,
let $K'$, $m$, $a_0,\ldots,a_m$ and $\alpha_0,\ldots,\alpha_m$ 
be as in Prop. \ref{prop:hofemrepres}.

The first part of the network in parallel emulates the CPwL ``hat'' functions
$\{ \theta^{\So}_{a_k} \}_{k=0}^{m}$
by subnetworks of depth $L_k \leq \ceil{\log_2(s(a_k))} + \ceil{\log_2(d+1)} + 7$
and numbers of neurons per layer 
$N_j \leq C (d+1) s(a_k) 2^{-j}$ for $j=1,\ldots,L_k$
for some constant $C>0$ independent of $\calT$.\footnote{
The value $C = 2^3$ stated in \cite[Lem. 4.2]{HX2023} gives $N_{L_k}<1$, 
which appears inconsistent.}
The parallelization of these $m+1$ networks
has depth 
$L \leq \max_{k=0}^{m} \ceil{\log_2(s(a_k))} + \ceil{\log_2(d+1)} + 7$
and numbers of neurons per layer
$N_j \leq C (d+1) 2^{-j} \sum_{k=0}^{m} s(a_k)$ for $j=1,\ldots,L$.

The second part of the network emulates the product of $p$ factors
stated in \eqref{eq:hofemrepres}.
In \cite[Lem. 4.2]{HX2023},
this is realized with $\ceil{\log_2(p)}$ layers.
Using a binary tree of product subnetworks,
the numbers $N_j$
of neurons per layer of this binary tree
are bounded as 
\footnote{
In \cite[Lem. 4.2]{HX2023},
the formula for the numbers of neurons per layer is incorrect,
as it does not take into account the numbers of neurons in these last $\ceil{\log_2(p)}$ layers.
In fact, for those layers, \cite[Eq. (4.11)]{HX2023} states that $N_j < 1$.}
\[
N_j \leq C p 2^{-j} \quad \mbox{for} \quad  j=1,\ldots,\ceil{\log_2(p)}.
\]

The concatenation of both parts has depth at most
\[
\max_{k=0}^{m} \ceil{\log_2(s(a_k))} + \ceil{\log_2(d+1)} + \ceil{\log_2(p)} + 7
\]
and numbers of neurons bounded by 
\[
N_j \leq C (d+1) 2^{-j} \sum_{k=0}^{m} s(a_k) 
\quad \mbox{for} \quad 
j=1,\ldots,\max_{k=0}^{m} \ceil{\log_2(s(a_k))} + \ceil{\log_2(d+1)} + 7
\]
and for the last $\ceil{\log_2(p)}$ layers
\[
N_{L-\ceil{\log_2(p)}+j} \leq C p 2^{-j} \quad \mbox{for} \quad  j=1,\ldots,\ceil{\log_2(p)},
\]
i.e.
$\max_{j=1}^L N_j \leq C \max\{ p, (d+1) \sum_{k=0}^{m} s(a_k) \}$.
For $i \in \calV$ it holds that $m = 0 \leq p$,
and for interpolation points that are not vertices,
by definition of $K'$ it holds that $i \in \interior{K'}$,
thus $i = \sum_{k=0}^{m} \alpha_k a_k / p$ 
for $\alpha_k\in\N$, i.e. none of the $\alpha_k$ vanishes.
From $\sum_{k=0}^{m} \alpha_k = p$ it thus follows that $m \leq p$.
By definition, it also holds that $m \leq d$,
thus we can estimate 
$\sum_{k=0}^{m} s(a_k) 
	\leq (m+1) \mathfrak{s}(\calV)
	\leq \min\{ d+1, p+1\} \mathfrak{s}(\calV)$.

To obtain an approximation of $v\in \Sop{p}(\calT,\domain)$,
it remains to take the linear combination of the Lagrangean basis functions 
for all interpolation points $i\in\calN$.
In \cite[Thm. 4.2]{HX2023}, 
this is done by placing these subnetworks in subsequent groups of layers,
increasing the depth and not the width.
In each layer, the numbers of neurons per layer is increased by $2d+1$ 
in order to forward the inputs to the hidden layers
and to keep track of the partial sums of the outputs of the subnetworks,
see \cite[Properties 4.2]{HX2023}.
This gives the bounds on the width and depth in Prop. \ref{prop:hxreluhofem}.
\end{remark}

\begin{remark}
\label{rem:hx2023complexity}
For fixed $\calT$ and $p\to\infty$, 
it holds that $\snorm\calN \simeq p^d$,
thus the total number of neurons is $\sum_{j=0}^L N_j = O(p^{d+1}\log(p))$,
which is larger than the number of degrees of freedom,
which equals $\snorm\calN \simeq p^d$.
The network size may be even larger,
as w.l.o.g. it is bounded from below by $\sum_{j=1}^L N_j$,
see Rem. \ref{rem:numberofneurons}.
\end{remark}
%
%%%%%%%%%%%%%%%%%%%%%%%%%%%%%%%%%%%%%%%%%%%%%%%%%%%%%%%%%%%%%%%%%%
\subsection{DNN Emulation of $hp$-FE spaces on Regular Triangulations}
\label{sec:efficienthofem}
%%%%%%%%%%%%%%%%%%%%%%%%%%%%%%%%%%%%%%%%%%%%%%%%%%%%%%%%%%%%%%%%%%
The NN emulation results in \cite{HX2023} 
focused on the NN size bounds 
for Lagrangean Finite Elements with nodal bases,
at fixed, uniform, polynomial degree $p\geq 1$.
Here, we focus on NN size bounds which are explicit 
in terms of $p\geq 1$ and tighter than those in \cite{HX2023}.
This is achieved by a suitably modified NN architecture
for the numerical realization of the shape functions.

Specifically, the emulation of all shape functions 
$\{ \theta^{\Sop{p}}_i \}_{i\in\calN}$
can be obtained more efficiently 
when we rearrange the part of the network that computes
the products in \eqref{eq:hofemrepres}.
Rather than computing for each Lagrangean basis function 
the product of $p$ factors from \eqref{eq:hofemrepres},
we first compute for all nodes $q\in\calV$ the 
real-valued quantities
\begin{align}
\label{eq:defwalpha}
w_\alpha( \theta^{\So}_{q}(x) ) 
	:= &\, \prod_{j=0}^{\alpha-1} 
	\left ( \tfrac{p}{j+1} \theta^{\So}_{q}(x) - \tfrac{j}{j+1} \right )
	\qquad
	\text{ for }\alpha=1,\ldots,p
	,
\end{align}
where $w_\alpha\in\bbP_\alpha \subset \bbP_p$.
For all $i\in\calN$,
Eq. \eqref{eq:hofemrepres} can then be rewritten as 
\begin{align}
\label{eq:hofemrepreswalpha}
\theta^{\Sop{p}}_i(x)
	= &\, 
		\prod_{k=0}^{m} w_{\alpha_{m}}( \theta^{\So}_{a_{m}}(x) )
.
\end{align}

\begin{remark}\label{rmk:pbds}
For all $i\in\calN$,
the number of factors in \eqref{eq:hofemrepreswalpha}
is ${m}+1 \leq d+1$,
one for each of the vertices $a_0,\ldots,a_{m}$.
This is independent of $p$
and leads to a network size of the order $O(p^d)$, 
rather than a number of neurons of the order 
$O(p^{d+1}\log(p))$ 
in \cite[Thm. 4.2]{HX2023} (see Rem. \ref{rem:hx2023complexity})
and an even larger network size.
\end{remark}

\begin{proposition}
\label{prop:reluhofem}
For all $p\in\N$ and a polytopal domain $\domain \subset \R^d$, 
for $2\leq d \in\N$, 
let $\calT$ be a regular, simplicial partition of $\domain$.

There exists a NN $\Phi^{CPwP_p} := \Phi^{CPwP_p(\calT,\domain)}$
with only \ReLU and \ReLUtwo activation
which in parallel emulates the shape functions 
$\{ \theta^{\Sop{p}}_i \}_{i\in \calN}$,
i.e. 
$\realiz{ \Phi^{CPwP_p} } : \domain \to \R^{\snorm\calN}$ 
satisfies
\begin{align*}
\realiz{ \Phi^{CPwP_p} }_i(x) = \theta^{\Sop{p}}_i(x), 
	\qquad
	\text{ for all }
	i\in\calN \text{ and } x\in\domain
	.
\end{align*}
The network depth and size satisfy
for some $C>0$ independent of $\calT$, $d$ and $p$
\begin{align*}
\depth(\Phi^{CPwP_p}) 
	\leq &\, C ( \log_2(p+1) + \log_2(d+1) + \log_2(\mathfrak{s}(\calV)) + 1 )
	,
	\\
\size(\Phi^{CPwP_p})
	\leq &\, C \Big( d \snorm\calN + p^2 \snorm\calV 
		+ \snorm{\calV} \log_2( \mathfrak{s}(\calV) ) + d^2 \sum_{q\in\calV} s(q) \Big)
	\\
	\leq &\, C \big( d \snorm\calN + p^2 \snorm\calV 
		+ d^2 \mathfrak{s}(\calV) \snorm\calV \big)
.
\end{align*}

For all
$v = \sum_{i\in\calN} v_i \theta^{\Sop{p}}_i \in\Sop{p}(\calT,\domain)$,
there exists a NN $\Phi^{CPwP_p,v} := \Phi^{CPwP_p(\calT,\domain),v} $
with only \ReLU and \ReLUtwo activations, such that for a constant $C>0$
independent of $\calT$, $d$ and $p$
\begin{align*}
\realiz{\Phi^{CPwP_p,v}}(x)
= &\, v(x)
\quad\text{ for all } x\in\domain
,
\\
\depth(\Phi^{CPwP_p,v})
\leq &\, C ( \log_2(p+1) + \log_2(d+1) + \log_2(\mathfrak{s}(\calV)) + 1 )
,
\\
\size(\Phi^{CPwP_p,v})
	\leq &\, C \Big( d \snorm\calN + p^2 \snorm\calV 
		+ \snorm{\calV} \log_2( \mathfrak{s}(\calV) ) + d^2 \sum_{q\in\calV} s(q) \Big)
	\\
	\leq &\, C \big( d \snorm\calN + p^2 \snorm\calV 
		+ d^2 \mathfrak{s}(\calV) \snorm\calV \big)
	.
\end{align*}

The layer dimensions and the lists of activation functions of
$\Phi^{CPwP_p}$ and $\Phi^{CPwP_p,v}$ are 
independent of $v$ and only
depend on $p$, and on $\calT$ through 
$\{m(i)\}_{i\in\calN}$, $\snorm\calN$,
$\{s(q)\}_{q\in\calV}$ and $\snorm\calV$.
In the first layers, only \ReLU activation is applied.
The number of these \ReLU layers is at most
$7 + \log_2( \mathfrak{s}(\calV) ) + \log_2(d+1)$.
The remaining hidden layers are strict \ReLUtwo layers, i.e., 
only \ReLUtwo activation is applied.
Their number is bounded by $C ( \log_2(d+1) + \log_2(p+1) )$ 
for an absolute constant $C>0$.

The set
$\calNN(CPwP_p;\calT,\domain) := \{ \Phi^{CPwP_p,v} : v \in
\Sop{p}(\calT,\domain) \}$ together with the linear operation
$\Phi^{CPwP_p,v} \widehat{+} \lambda\Phi^{CPwP_p,w} := \Phi^{CPwP_p,v+\lambda w}$
for all $v,w\in \Sop{p}(\calT,\domain)$ and all $\lambda\in\R$ is a
vector space.
The realization map
$\realiz{\cdot}: \calNN(CPwP_p;\calT,\domain) \to \Sop{p}(\calT,\domain)$
is a linear isomorphism.
\end{proposition}

\begin{proof}
For each $p\in\N$ and each interpolation point $i\in\calN$,
let 
$m$,
$a_0,\ldots,a_{m}$
and
$\alpha_0,\ldots,\alpha_{m}$
be as in Prop. \ref{prop:hofemrepres}.
In the remainder of this proof, we will denote them by 
${m}(i)$,
$a_0(i),\ldots,a_{{m}(i)}(i)$
and
$\alpha_0(i),\ldots,\alpha_{{m}(i)}(i)$.

This proof consists of three steps.
In Step 1, we construct $\Phi^{CPwP_p}$ and prove the formula for its realization.
in Step 2, we prove the bounds on the network depth and size of $\Phi^{CPwP_p}$.
The remainder of the statements is proved in Step 3.

\textbf{Step 1.}
First, we recall several NNs which we will use in our construction.
From Prop. \ref{prop:CPLbasisnet} we recall 
$\Phi^{CPwL} := \Phi^{CPwL(\calT,\domain)}$,
whose depth and size are bounded by 
$$
\depth(\Phi^{CPwL}) \leq 8 + \log_2( \mathfrak{s}(\calV) ) + \log_2(d+1),
\qquad
\size(\Phi^{CPwL}) \leq \snorm{\calV} \log_2( \mathfrak{s}(\calV) ) + d^2 \sum_{q\in\calV} s(q).
$$

We consider NN emulations of the polynomials $\{ w_\alpha \}_{\alpha=1}^p \in \bbP_p$ 
defined in \eqref{eq:defwalpha}.
By sparsely concatenating
these networks $\{ \Phi^{w_\alpha} \}_{\alpha=1}^p$
with a \ReLUtwo identity network from Prop. \ref{prop:relutwoidnn}
using Prop. \ref{prop:relutwoconcat},
we obtain that for all $p\in\N$ and $\alpha=1,\ldots,p$,
there exists a \ReLUtwo NN $\Phi^{w_\alpha}_p$ such that 
$\realiz{\Phi^{w_\alpha}_p} = w$,
$\depth(\Phi^{w_\alpha}_p) = c \ceil{\log_2(p+1)}$ 
and 
$\size(\Phi^{w_\alpha}_p) \leq C p$,
for constants $c\in\N$ and $C>0$ independent of $p$ and $\alpha=1,\ldots,p$.
We now prove this.

Let $c\in\N$ be such that
$\depth( \Phi^{w_\alpha} ) < c \ceil{\log_2(p+1)} $ for all $\alpha=1,\ldots,p$.
With 
$L_\alpha 
	:= c \ceil{\log_2(p+1)} - \depth( \Phi^{w_\alpha} ) 
	\leq c \ceil{\log_2(p+1)}$
we then define 
$\Phi^{w_\alpha}_p 
	:= \Phi^{\Id,\relu^2}_{1,L_\alpha} \sconc_{\relu^2} \Phi^{w_\alpha}$
for all $\alpha=1,\ldots,p$.
We obtain that 
$\realiz{ \Phi^{w_\alpha}_p } = \realiz{ \Phi^{w_\alpha} } = w_\alpha$,
$\depth( \Phi^{w_\alpha}_p )
	= L_\alpha + \depth( \Phi^{w_\alpha} ) = c \ceil{\log_2(p+1)}$
and
$\size( \Phi^{w_\alpha}_p )
	\leq C \size( \Phi^{\Id,\relu^2}_{1,L_\alpha} ) + C \size( \Phi^{w_\alpha} )
	\leq C L_\alpha + C (\alpha+1)
	\leq C p$,
for a constant $C>0$ independent of $\alpha$ and $p$.

Analogously, we obtain from Prop. \ref{prop:nnprod} the following result.
For all $d\in\N$,
there exist \ReLUtwo NNs $\{ \Phi^{\mathrm{prod}}_{\ell,d+1} \}_{\ell=1}^{d+1}$ 
such that 
$\realiz{\Phi^{\mathrm{prod}}_{\ell,d+1}}(x_1,\ldots,x_\ell) = \prod_{j=1}^\ell x_j$ 
for all $x_1,\ldots,x_\ell\in\R$,
and such that their NN depths and sizes satisfy
$\depth(\Phi^{\mathrm{prod}}_{\ell,d+1}) = c \ceil{\log_2(d+1)}$ 
and 
$\size(\Phi^{\mathrm{prod}}_{\ell,d+1}) \leq C d$,
for constants $c\in\N$ and $C>0$ independent of $d$.

We can now define
\begin{align*}
\Phi^{CPwP_p} := &\, 
	\Phi^{CPwP_p}_{(1)} \sconc_{\relu^2} \Phi^{CPwP_p}_{(2)} \sconc_{\relu^2} \Phi^{CPwP_p}_{(3)} \sconc_{\relu^2} \Phi^{CPwP_p}_{(4)}
	\\
	:= &\, \FParallelc{ \left \{ \Phi^{\mathrm{prod}}_{{m}(i)+1,d+1} \right \}_{i\in\calN} } 
	\sconc_{\relu^2} \Phi^{CPwP_p}_{(2)}
	\sconc_{\relu^2} \FParallelc{ \left \{ 
	\Parallelc{ \left \{ \Phi^{w_\alpha}_p \right \}_{\alpha=1}^p } 
	\right \}_{q\in\calV} }
	\sconc_{\relu^2} \Phi^{CPwL}
	,
\end{align*}
where
$\Phi^{CPwP_p}_{(2)}$ is a NN of depth $1$ emulating a linear transformation
$x\mapsto Ax$.
Fixing a bijection $n:\calN\to\{1,\ldots,\snorm\calN\}$,
the weight matrix $A$ 
of size $\big( \sum_{i\in\calN} ( m(i)+1 ) \big) \times p \snorm\calV$
is defined by
\begin{align*}
A_{{j_1},{j_2}} 
= 
\begin{cases}
1 & \text{ if }
j_1  = k + 1 + \sum_{n'=1}^{n(i)-1} ( {m}(n^{-1}(n')) +1 ) 
\text{ for some }k\in\{0,\ldots,m(i)\}
\\
& \text{ and } w_{\alpha_k(i)}(\theta^{\So}_{a_k(i)}(x))
\text{ is the } j_2\text{'th component of } \realiz{ \Phi^{CPwP_p}_{(3)} \sconc_{\relu^2} \Phi^{CPwP_p}_{(4)} },
\\
0 & \text{ else}.
\end{cases} 
\end{align*}
Note that each row has precisely one nonzero element.
The number of rows is bounded by $(d+1) \snorm\calN$,
hence
$\size(\Phi^{CPwP_p}_{(2)}) \leq (d+1) \snorm\calN$.

It follows directly from these definitions that
for all $x\in\R^d$
\begin{align*}
\realiz{ \Phi^{CPwP_p}_{(4)} }(x) 
	= &\, \{ \theta^{\So}_q(x) \}_{q\in\calV}
	,
	\\
\realiz{ \Phi^{CPwP_p}_{(3)} \sconc_{\relu^2} \Phi^{CPwP_p}_{(4)} }(x) 
	= &\, \{ \{ w_\alpha(\theta^{\So}_q(x)) \}_{\alpha=1}^p \}_{q\in\calV}
	,
	\\
\realiz{ \Phi^{CPwP_p}_{(1)} \sconc_{\relu^2} \Phi^{CPwP_p}_{(2)} \sconc_{\relu^2} \Phi^{CPwP_p}_{(3)} \sconc_{\relu^2} \Phi^{CPwP_p}_{(4)} }_i(x) 
	= &\, \prod_{k=0}^{{m}(i)} w_{\alpha_k(i)}( \theta^{\So}_{a_k(i)}(x) ) 
	= \theta^{\Sop{p}}_i(x)
	.
\end{align*}

\textbf{Step 2.}
To obtain an \emph{estimate on the network depth},
we combine the bounds from Step 1 on the depths of the subnetworks
using Propositions \ref{prop:parallSep} and \ref{prop:relutwoconcat},
and get
\begin{align*}
\depth( \Phi^{CPwP_p} )
	\leq &\, \depth( \Phi^{CPwP_p}_{(1)} ) + \depth( \Phi^{CPwP_p}_{(2)} ) + \depth( \Phi^{CPwP_p}_{(3)} ) + \depth( \Phi^{CPwP_p}_{(4)} )
	\\
	\leq &\, c \ceil{\log_2(d+1)} + 1 + C \ceil{\log_2(p+1)} 
		+ \big( 8 + \log_2( \mathfrak{s}(\calV) ) + \log_2(d+1) \big)
	\\
	\leq &\, C \big( 1 + \log_2( \mathfrak{s}(\calV) ) + \log_2(d+1) + \log_2(p+1) \big)
.	
\end{align*}
Similarly, using the same propositions to
combine bounds on the sizes of the subnetworks,
we obtain as
\emph{bound on the network size}
\begin{align*}
\size( \Phi^{CPwP_p} )
	\leq &\, C \size( \Phi^{CPwP_p}_{(1)} ) + C \size( \Phi^{CPwP_p}_{(2)} ) + C \size( \Phi^{CPwP_p}_{(3)} ) 
		+ C \size( \Phi^{CPwP_p}_{(4)} )
	\\
	\leq &\, C (d+1) \snorm\calN + C (d+1) \snorm\calN + C p^2 \snorm\calV
		+ C \Big ( \snorm{\calV} \log_2( \mathfrak{s}(\calV) ) + d^2 \sum_{q\in\calV} s(q) \Big )
	\\
	\leq &\, C \Big( d \snorm\calN + p^2 \snorm\calV 
		+ \snorm{\calV} \log_2( \mathfrak{s}(\calV) ) + d^2 \sum_{q\in\calV} s(q) \Big)
	\\
	\leq &\, C \big( d \snorm\calN + p^2 \snorm\calV 
		+ d^2 \mathfrak{s}(\calV) \snorm\calV \big)
.	
\end{align*}
The subnetwork $\Phi^{CPwP_p}_{(4)} = \Phi^{CPwL}$ only has \ReLU activations
and comprises at most $7 + \log_2( \mathfrak{s}(\calV) ) + \log_2(d+1)$ hidden layers,
which is one less than the total number of layers.
The subnetworks $\Phi^{CPwP_p}_{(1)}$, $\Phi^{CPwP_p}_{(2)}$, $\Phi^{CPwP_p}_{(3)}$ 
and the sparse concatenations have only \ReLUtwo activations.
The number of such hidden layers is at most 
$\depth( \Phi^{CPwP_p}_{(1)} ) 
	+ \depth( \Phi^{CPwP_p}_{(2)} ) 
	+ \depth( \Phi^{CPwP_p}_{(3)} )
	\leq C ( \log_2(d+1) + \log_2(p+1) )$.

\textbf{Step 3.}
For all $v\in\Sop{p}(\calT,\domain)$,
the NN $\Phi^{CPwP_p,v}$ is defined as
$(A',0,\Id_{\R}) \bullet \Phi^{CPwP_p}$,
where $A' \in \R^{1\times \snorm\calN}$ 
is the row vector containing the weights $A'_{1,i} = v_i$ for all $i\in\calN$.
It holds that 
$\depth( \Phi^{CPwP_p,v} ) 
	= 1 - 1 + \depth( \Phi^{CPwP_p} ) 
	= \depth( \Phi^{CPwP_p} )$.

To estimate its size,
we observe that 
the hidden layers of $\Phi^{CPwP_p}$ and $\Phi^{CPwP_p,v}$ coincide,
as these layers emulate the FE basis functions.
The layer dimensions and the lists of activation functions of $\Phi^{CPwP_p}_{(4)}$ 
depend on $\calT$ only through $\{ s(q) \}_{q\in\calV}$ and $\snorm\calV$.
Those of $\Phi^{CPwP_p}_{(3)}$ only depend on $p$ and $\snorm\calV$,
and those of $\Phi^{CPwP_p}_{(1)}$ and $\Phi^{CPwP_p}_{(2)}$
only on $p$, $\snorm\calV$, $\snorm\calN$ and $\{ m(i) \}_{i\in\calN}$.
Furthermore,
each weight in the output layer of $\Phi^{CPwP_p,v}$
is the inner product of the row vector $A'$ 
with a column of the output layer weight matrix of $\Phi^{CPwP_p}$,
hence the number of nonzero weights of $\Phi^{CPwP_p,v}$ 
is at most that of $\Phi^{CPwP_p}$.
The same holds for the number of nonzero biases,
as the bias of $\Phi^{CPwP_p,v}$ 
is the inner product of $A'$ with the bias vector of $\Phi^{CPwP_p}$.
Thus, it follows that 
\[
\size( \Phi^{CPwP_p,v} ) \leq \size( \Phi^{CPwP_p} ).
\]

By definition of $\Phi^{CPwP_p,v}$, 
the realization $\realiz{ \Phi^{CPwP_p,v} }$ 
is a linear combination of the outputs of $\realiz{ \Phi^{CPwP_p} }$,
which is determined uniquely by the weights in the output layer,
$( v_i )_{i\in\calN}$ (which coincide, due to \eqref{eq:hofemrepres}, 
with the nodal values of $v$).
Therefore, 
$\realiz{\cdot}: \calNN(CPwP_p;\calT,\domain)$ $\to \Sop{p}(\calT,\domain)$
is a bijection. 
With the linear operations defined in the proposition,
this map is linear by definition, thus a linear isomorphism.
\end{proof}

\begin{remark}
\label{rem:sizeproptodof}
Using that $\snorm\calV \leq \sum_{q\in\calV} s(q) \leq (d+1) \snorm\calT$
by Rem. \ref{rem:sizebound}
and that $\snorm\calN \simeq p^d \snorm\calT$
with a proportionality constant which depends on $d$,
we obtain that the network size is bounded by 
$C(d) ( \log_2( \mathfrak{s}(\calV) ) + p^d ) \snorm\calT$.

For a fixed triangulation $\calT$, for increasing polynomial degree this scales as $O(p^d)$,
which significantly improves 
the complexity of the networks from \cite[Thm. 4.2]{HX2023},
cf. Rem. \ref{rmk:pbds} above.
Actually, it is of the same order of $p$ as 
the number of degrees of freedom 
of the high-order finite elements which it emulates.

When we do not consider a single, fixed partition, 
but rather a family of regular, simplicial partitions of $\domain$ 
which is uniformly $\kappa$-shape regular,
the value $\mathfrak{s}(\calV)$ is bounded from above in terms of $\kappa$,
according to Rem. \ref{rem:mathfraks}.
Using again that 
$\snorm\calV \leq (d+1) \snorm\calT$
and
$\snorm\calN \simeq p^d \snorm\calT$,
it follows that
\begin{align*}
\size( \Phi^{CPwP_p} )
	\leq &\, C \big( d \snorm\calN + p^2 \snorm\calV 
		+ d^2 \mathfrak{s}(\calV) \snorm\calV \big)
	\leq
		C(d,\kappa) \big( p^d \snorm\calT + p^2 \snorm\calT 
		+ \snorm\calT \big)
	\leq 
		C(d,\kappa) \snorm\calN
	.
\end{align*}
\end{remark}

\begin{remark}
\label{rem:nodalinterpolation}
Because the outputs of $\Phi^{CPwP_p}$
are continuous, piecewise polynomial Lagrangean basis functions
which vanish in all nodes in $\calN$ except one,
it follows that the coefficients $(v_i)_{i\in\calN}$ 
are equal to the function values in $\calN$,
i.e. $v_i = v(i)$ for all $i\in\calN$.
In particular,
the construction of $\Phi^{CPwP_p,v}$ is well suited 
for obtaining NN approximations based on
high order finite element interpolation.
\end{remark}

\begin{remark}
[Univariate case $d=1$]
\label{rem:oned}
Although the proposition is stated for $d\geq 2$,
the NN construction in the proof of Prop. \ref{prop:reluhofem} 
also holds for $d=1$.
However, the proposition does not hold as stated in the univariate case $d=1$,
because the parallel emulation of $p$ univariate polynomials of degree $p$
which is used in the proof
has network size bound $C p^2$, which is not linear in the polynomial degree.

In the univariate case, when $d=1$,
an exact NN emulation with size $\snorm\calN \simeq p \snorm\calT$ 
can be constructed as follows.
In \cite[Appendices A and B]{OSX2024},
for $p\in\N$ 
we constructed a NN with one input and $p$ outputs,
of depth $\ceil{\log_2(p)}+1$ and size at most $C p$,
whose outputs approximate the \Cheb polynomials of degree $1,\ldots,p$.
While \cite[Prop. A.2]{OSX2024} is stated for $\tanh$-NNs,
by \cite[Rem. B.4]{OSX2024} it also holds for \ReLUtwo NNs.
In the proof of \cite[Prop. A.2]{OSX2024},
approximations of the \Cheb polynomials
are constructed using subnetworks,
among which subnetworks that approximate the identity
with a network of fixed size, independently of the accuracy.
Similarly, subnetworks approximating the product of two numbers are used
whose size is fixed and independent of the desired accuracy.
When we replace those by the product and identity networks 
from Prop. \ref{prop:relutwoidnn} and \ref{prop:nnprod},
we obtain a \ReLUtwo NN which exactly emulates the \Cheb polynomials of degrees $1,\ldots,p$
with depth $\ceil{\log_2(p)}+1$ and network size at most $Cp$.
This result can be used in the constructions from \cite{OS2023}
instead of \cite[Prop. 4.8]{OS2023}.
Following the steps in the proofs of 
\cite[Cor. 4.10 and Prop. 4.11]{OS2023}
then gives an exact \ReLUnospace- and \ReLUtwonospace-activated NN emulation of 
univariate continuous, piecewise polynomial functions
on a partition of a bounded interval into $\snorm\calT$ elements
with elementwise polynomial degree $p$,
whose depth is of the order $O(\log(p+1))$
and with a network size of the order $O( p \snorm\calT )$.
\end{remark}

%%%%%%%%%%%%%%%%%%%%%%%%%%%%%%%%%%%%%%%%%%%%%%%%%%%%%%%%%%%%%%%%%%
\section{NN Approximation of Weighted Gevrey Regular Functions}
\label{sec:nnhpApprox}
%%%%%%%%%%%%%%%%%%%%%%%%%%%%%%%%%%%%%%%%%%%%%%%%%%%%%%%%%%%%%%%%%%
The main result is now obtained 
as a direct consequence of Prop. \ref{prop:eta} and \ref{prop:reluhofem}.

\begin{theorem}
\label{thm:nngevrey}
In a bounded polytope $\domain\subset \mathbb{R}^d$, $d=2,3$,
with plane sides resp. faces,
suppose given a weight vector $\wbeta$ as in~\eqref{eq:beta} 
if $d=3$ or \eqref{eq:beta2} if $d=2$.

Then, 
for all $\delta>0$,
there exist constants
$b,\tilde{b},C>0$
depending on $C_u$, $d_u$ in \eqref{eq:Abeta},
independent of $p$,
such that
for every 
$u\in \calG^\delta_{\wbeta}(\S;\domain)$ 
and $p\in\N$,
there exists a NN $\Phi^{hp,u,p}$
with only \ReLU and \ReLUtwo activation,
which satisfies 
\begin{align}
\depth(\Phi^{hp,u,p})
	\leq &\, C \log_2(p+1)
	,
	\qquad
\size(\Phi^{hp,u,p})
	\leq 
		C {p^{d+1/\delta}}
	,
	\nonumber\\
	\label{eq:nngevrey}
\NN{u - \realiz{\Phi^{hp,u,p}} }_{H^1(\domain)} 
	\leq &\, 
\begin{cases} \;
C\exp(-b p^{1/\delta}) 
\leq C\exp(-\tilde{b} \size(\Phi^{hp,u,p})^{\frac{1}{1+\delta d}} ) 
&\delta\geq 1,
\\
\; 
C\left(\Gamma\left(p^{1/\delta}\right)\right)^{-b(1-\delta)} 
\leq 
C\left(\Gamma\left(\size(\Phi^{hp,u,p})^{\frac{1}{1+\delta d}}\right)\right)^{-\tilde{b}(1-\delta)} 
&0<\delta<1.
\end{cases}
\end{align}
In addition, 
if $u|_{\partial \domain} = 0$, 
then
$\realiz{\Phi^{hp,u,p}} |_{\partial \domain} = 0$.
\end{theorem}

\begin{proof}
For some $0<\sigma<1$ and $\kappa>1$,
consider a sequence $\MMM_{\kappa,\sigma}(\S)$ of nested,
regular simplicial meshes in $\domain$
which are $\sigma$-geometrically refined towards $\S$
and which are $\kappa$-shape regular,
e.g. those constructed in \cite[Alg. 1]{hpfem}.
By Prop. \ref{prop:eta},
for all $p\in\N$,
with $\ell\simeq p^{1/\delta}$,
the continuous, piecewise polynomial function
$v := \Pi^{p}_{\kappa,\sigma} u \in S^p(\M^{(\ell)})$
satisfies
\begin{equation*}
\NN{u - v }_{H^1(\domain)} 
\leq C 
\begin{cases} \;\exp(-bN^{\frac{1}{1+\delta d}}) &\delta\geq 1,
\\
\; 
\left(\Gamma\left(N^{\frac{1}{1+\delta d}}\right)\right)^{-b(1-\delta)} &0<\delta<1,
\end{cases}
\end{equation*}
where
$
N={\rm dim}(S^p(\M^{(\ell)})) \simeq p^{d + {1/\delta}}.
$
If $u|_{\partial \domain} = 0$, 
then $v|_{\partial \domain} = 0$.

Now, by Prop. \ref{prop:reluhofem}
there exists 
$\Phi^{hp,u,p}
	:= \Phi^{CPwP_p(\M^{(\ell)},\domain),v}$,
with only \ReLU and \ReLUtwo activations, 
which satisfies for a constant $C>0$
independent of $\M^{(\ell)}$, $d$ and $p$
\begin{align*}
\realiz{\Phi^{CPwP_p(\M^{(\ell)},\domain),v}}(x)
= &\, v(x)
\quad\text{ for all } x\in\domain
,
\\
\depth(\Phi^{CPwP_p(\M^{(\ell)},\domain),v})
\leq &\, C ( \log_2(p+1) + \log_2(d+1) + \log_2(\mathfrak{s}(\calV)) + 1 )
,
\\
\size(\Phi^{CPwP_p(\M^{(\ell)},\domain),v})
	\leq &\, C \big( d \snorm\calN + p^2 \snorm\calV 
		+ d^2 \mathfrak{s}(\calV) \snorm\calV \big)
	.
\end{align*}
By Rem. \ref{rem:mathfraks},
$\mathfrak{s}(\calV)$ can be bounded in terms of the shape regularity constant $\kappa$,
which is bounded independently of $p$.
Also, $d\in\{2,3\}$ is bounded independently of $p$,
which shows that 
$\depth(\Phi^{hp,u,p}) \leq C \log_2(p+1)$ for $C>0$ depending on $\kappa$.
To estimate the network size,
we recall from Rem. \ref{rem:sizeproptodof} that
$\size(\Phi^{hp,u,p}) 
	= \size(\Phi^{CPwP_p(\M^{(\ell)},\domain),v})
	\leq C \snorm\calN 
	= C N
	\simeq p^{d+1/\delta}$
for a constant $C$ depending on $\kappa$, independent of $p$.
To finish the proof, we substitute this into the error bound.
\end{proof}

\begin{remark}
\label{rem:bc}
The treatment of homogeneous Dirichlet boundary conditions
on a strict subset $\Gamma_{\rm dir}\subset\partial\domain$
comprising the union of several boundary faces (when $d=3$) or edges (when $d=2$)
is the topic of \cite[Sect. 4.2.7]{hpfem}.
Because Prop. \ref{prop:eta} is based on nodal interpolation,
it also holds that 
$u|_{\Gamma_{\rm dir}} = 0$
implies that
$\Pi^p_{\kappa,\sigma} u|_{\Gamma_{\rm dir}} = 0$.
By the presently developed exact NN emulation, this implies that
$\realiz{\Phi^{hp,u,p}} |_{\Gamma_{\rm dir}} = 0$.
\end{remark}

%%%%%%%%%%%%%%%%%%%%%%%%%%%%%%%%%%%%%%%%%%%%%%%%%%%%%
\section{Conclusions}
\label{sec:Concl}
%%%%%%%%%%%%%%%%%%%%%%%%%%%%%%%%%%%%%%%%%%%%%%%%%%%%%%%%%%%%%%%%%%
%
For DNNs with a suitable combination of \ReLU and \ReLUtwo activations,
we established exponential expression rate bounds \eqref{eq:nngevrey}
in polytopal domains $\domain \subset \R^d$, in dimension $d=2,3$.
The results built on a general result on DNN emulation of $hp$-FE spaces,
Prop.~\ref{prop:reluhofem}, which is of independent interest,
also allowing for DNN emulation of the so-called $p$-version 
Finite Element Method, and Spectral Element Methods in $\domain$.
Prop.~\ref{prop:reluhofem} improves earlier results from 
\cite{HX2023} by realizing the same finite element
with NNs which are smaller,
i.e. the network size admits smaller bounds in terms of the polynomial degree $p$
(cf. Rem.~\ref{rmk:pbds}).
With Prop.~\ref{prop:reluhofem} in place, exponential expression rate bounds
for the constructed DNNs on countably normed, corner-weighted function classes
\eqref{eq:Abeta} follow from known results on $hp$-FE approximation 
e.g. in \cite{hpfem}, which go back to the $hp$-FE approximation rates
by I.M. Babu\v{s}ka and B.Q. Guo in \cite{GuoDiss85,BabGuoCurved,ApproxhpFE}.
The present results also imply further ``DNN-versions'' of $hp$-FE Methods, 
such as \cite{BMS23-SpcFrLap,BMS23-reactdiff,FMMS23}.

The proposed DNN emulation of Lagrangean basis functions
of the space of polynomials of total degree $p\geq 1$ 
in equispaced nodes 
(in barycentric coordinates in simplicial elements $K\in \calT$)
is, as was shown in the proof of Prop.~\ref{prop:reluhofem},
mathematically exact. 
The well-known conditioning issues of these
nodal basis functions may preclude their 
training in finite-precision arithmetic for very high orders.
Here, the \emph{modal bases} such as those proposed by I.M. Babu\v{s}ka
and coworkers, which are based on antiderivatives of univariate 
Legendre polynomials, will afford better conditioning of the 
resulting ``feature space'', but do not exhibit the separability 
which we used in the proof of Prop.~\ref{prop:reluhofem}. 
Other polynomial bases, both nodal and modal, with substantially
better conditioning for large polynomial orders, are known
(e.g. \cite{ChenIMBPts96,HestPts98}). 
\emph{Separable, nodal bases} are provided by Dubiner's approach \cite{Dubiner},
which also yields bases whose conditioning 
is improved w.r. to the Lagrangean basis \eqref{eq:defwalpha} 
used in the proof of Prop.~\ref{prop:reluhofem}.
The NN size bounds for DNN emulation of these bases is the topic of 
future work.

%%%%%%%%%%%%%%%%%%%%%%%%%%%%%%%%%%%%%%%%%%%%%%%%%%%%%%%%%%%%%%%%%%
%\subsection*{Acknowledgement} \label{sec:Ackn}
%%%%%%%%%%%%%%%%%%%%%%%%%%%%%%%%%%%%%%%%%%%%%%%%%%%%%%%%%%%%%%%%%%

{\small
%\bibliography{bibliography}

}

\end{document}